\documentclass[a4paper]{amsart}
\usepackage{amscd}
\usepackage{amsmath}
\usepackage{amssymb}
\usepackage{amsthm}
\usepackage{bbm}
\usepackage{stmaryrd}

\usepackage[T1]{fontenc}

\newif\ifpdf
\ifx\pdfoutput\undefined
   \pdffalse        % we are not running PDFLaTeX
\else
   \pdfoutput=1     % we are running PDFLaTeX
   \pdftrue
\fi

\ifpdf
   \usepackage[pdftex]{graphicx}
\else
   \usepackage{graphicx}
\fi

\numberwithin{equation}{section} \swapnumbers

\newtheorem{satz}{Satz}[section]

\newtheorem{theorem}[satz]{Theorem}
\newtheorem{proposition}[satz]{Proposition}
\newtheorem{corollary}[satz]{Corollary}
\newtheorem{lemma}[satz]{Lemma}
\newtheorem{assumption}[satz]{Assumption}

\newtheorem{definition}[satz]{Definition}

\newtheorem{remark}[satz]{Remark}

\newtheorem{example}[satz]{Example}

\newcommand{\bbr}{\mathbb{R}}

\newcommand{\bbn}{\mathbb{N}}

\newcommand{\bbj}{\mathbb{J}}

\newcommand{\bbt}{\mathbb{T}}
\newcommand{\bbm}{\mathbb{M}}
\newcommand{\bbs}{\mathbb{S}}

\newcommand{\cald}{\mathcal{D}}
\newcommand{\calk}{\mathcal{K}}
\newcommand{\call}{\mathcal{L}}
\newcommand{\calm}{\mathcal{M}}

\newcommand{\frc}{\mathfrak{C}}
\newcommand{\fri}{\mathfrak{I}}

\begin{document}

\hyphenation{rea-li-za-tion pa-ra-me-tri-za-ti-ons pa-ra-me-tri-za-ti-on}

\title[Affine realizations with affine state processes]{Affine realizations with affine state processes for stochastic partial differential equations}
\author{Stefan Tappe}
\address{Leibniz Universit\"{a}t Hannover, Institut f\"{u}r Mathematische Stochastik, Welfengarten 1, 30167 Hannover, Germany}
\email{tappe@stochastik.uni-hannover.de}
\thanks{I am grateful to Ozan Akdogan, Darrell Duffie, Damir Filipovi\'c and Matthias Sch\"{u}tt for valuable comments and discussions.}
\thanks{I am also grateful to an anonymous referee for the careful study of my paper and the valuable comments and suggestions.}
\begin{abstract}
The goal of this paper is to clarify when a stochastic partial differential equation with an affine realization admits affine state processes. This includes a characterization of the set of initial points of the realization. Several examples, as the HJMM equation from mathematical finance, illustrate our results.
\end{abstract}
\maketitle

\section{Introduction}\label{sec-intro}

The goal of this paper is to clarify when a semilinear stochastic partial differential equation (SPDE) of the form
\begin{align}\label{SPDE-manifold}
\left\{
\begin{array}{rcl}
dr_t & = & (A r_t + \alpha(r_t))dt + \sigma(r_{t})dW_t
\medskip
\\ r_0 & = & h_0
\end{array}
\right.
\end{align}
in the spirit of \cite{Da_Prato} driven by a $\bbr^n$-valued Wiener process $W$ (for some positive integer $n \in \bbn$) with an affine realization admits affine and admissible state processes. Affine realizations are particular types of finite dimensional realizations (FDRs). Denoting by $H$ the state space of (\ref{SPDE-manifold}), which we assume to be a separable Hilbert space, the idea of a FDR is that for each starting point $h_0 \in \fri$ (where $\fri \subset H$ denotes the set of initial points) we can express the weak solution $r$ to (\ref{SPDE-manifold}) locally as
\begin{align}\label{FDR}
r = \varphi(X)
\end{align}
for some $\bbr^d$-valued (typically time-inhomogeneous) process $X$ and a deterministic mapping $\varphi : \bbr^d \to H$, which makes the infinite dimensional SPDE (\ref{SPDE-manifold}) more tractable. If we have a representation of the form (\ref{FDR}), then the mapping $\varphi$ is the parametrization of an invariant submanifold $\calm$.

In this situation, the term \textit{affine} has a twofold meaning, which we shall now explain. We speak about an affine realization if for each starting point $h_0  \in \fri$ we can express the weak solution $r$ to (\ref{SPDE-manifold}) locally as
\begin{align}\label{decomp-intro}
r = \psi + X
\end{align}
with a deterministic curve $\psi : \bbt \to H$, where $\bbt = [0,\delta]$ for some $\delta > 0$, and a process $X$ having values in a state space of the form $\frc \oplus U$ with a finite dimensional proper cone $\frc \subset H$ and a finite dimensional subspace $U \subset H$. In this case, we also say that the SPDE (\ref{SPDE-manifold}) has an affine realization generated by $\frc \oplus U$, and the invariant manifold $(\calm_t)_{t \in \bbt}$ is a collection of affine spaces 
\begin{align}\label{foliations}
\calm_t = \psi(t) + \frc \oplus U, \quad t \in \bbt,
\end{align}
also called a foliation, and the curve $\psi$ is a parametrization of $(\calm_t)_{t \in \bbt}$.

We say that such an affine realization has affine and admissible state processes if for each starting point $h_0 \in \fri$ the process $X$ appearing in (\ref{decomp-intro}) is a (typically time-inhomogeneous) affine and admissible process on the state space $\frc \oplus U$. Here the term \emph{affine} means that the local characteristics of $X$ are affine, that is, the drift is affine and the volatility is square-affine, and the term \emph{admissible} means that the state space $\frc \oplus U$ is invariant for $X$, which means that the drift is inward pointing and the volatility is parallel to the boundary at boundary points of $\frc \oplus U$.\footnote[1]
{In the literature, a process is usually called an affine process if it is affine and admissible in the just described sense. For the purposes of this paper, we will carefully distinguish between the terms \emph{affine} and \emph{admissible}.}

There is a substantial literature about FDRs for SPDEs, in particular for the HJMM equation from mathematical finance. Here we use the name HJMM equation, as it is the Heath-Jarrow-Morton (HJM) model from \cite{HJM} with Musiela parametrization presented in \cite{Musiela}. The existence of FDRs for the HJMM equation driven by Wiener processes has intensively been studied in the literature, and we refer to \cite{Bj_Sv, Bj_La, Filipovic-Teichmann, Filipovic-Teichmann-royal} and references therein, and to \cite{Bjoerk} for a survey. As shown in \cite{Filipovic-Teichmann}, the existence of a FDR for the Wiener process driven HJMM equation implies the existence of an affine realization. The existence of affine realizations has been studied in \cite{Tappe-Wiener} for the HJMM equation driven by Wiener processes, in \cite{Tappe-Levy, Platen-Tappe} for the HJMM equation driven by L\'{e}vy processes, and in \cite{Tappe-affin-real} for general SPDEs driven by L\'{e}vy processes.

Affine processes have found growing interest due to their analytical tractability, in particular regarding applications in the field of mathematical finance. We refer, e.g., to \cite{Duffie-Kan, Filipovic-affine, DFS, Filipovic-inh, Filipovic-Mayerhofer} for affine processes on the canonical state space, and, e.g., to \cite{Cuchiero-et-al, Cuchiero-et-al-2, Spreij} for affine processes on more general state spaces. We also mention the recent papers \cite{Biagini} and \cite{Cu-Fo-Gn}, where HJM-type models driven by affine processes are studied. Note that our state space $\frc \oplus U$ corresponds to the canonical state space $\bbr_+^m \times \bbr^{d-m}$.

The goal of this paper is to clarify when the SPDE (\ref{SPDE-manifold}) admits an affine realization with affine and admissible state processes -- which has not been studied in the literature so far -- and to derive conditions on the parameters $(A,\alpha,\sigma)$ of (\ref{SPDE-manifold}) and on the set $\fri$ of initial points, which are necessary and sufficient. This includes a characterization of the structure of the set $\fri$, which we will use in order to construct this set for concrete examples.  

In order to outline the main results of this paper, let us first discuss how for a given invariant foliation $(\calm_t)_{t \in \bbt}$ the affine and admissibility properties of the state process $X$ appearing in (\ref{decomp-intro}) can be characterized by means of $(A,\alpha,\sigma)$; we refer to Section~\ref{sec-foliations} and Appendix~\ref{app-affine} for further details and the precise statements. Let $G \subset H$ be a closed subspace such that we have a direct decomposition $H = G \oplus V$ of the Hilbert space, where $V = C \oplus U$ and $C = \langle \frc \rangle$, the linear space generated by the cone.\footnote[2]{Later, the subspace $G$ will be uniquely determined by the set $\fri$ of initial points.} Without loss of generality, we may assume that the parametrization $\psi$ has its values in $G$, that is $\psi \in C^1(\bbt;G)$. Since the foliation $(\calm_t)_{t \in \bbt}$ is invariant for the SPDE (\ref{SPDE-manifold}), we obtain the well-known tangential conditions
\begin{align}\label{tang-domain}
\bbm &\subset \cald(A),
\\ \label{tang-beta} \beta(h) &\in T \calm_t \quad \text{for all $t \in \bbt$ and all $h \in \calm_t$,}
\\ \label{tang-sigma} \sigma(\bbm) &\subset V^n,
\end{align}
where $\bbm = \bigcup_{t \in \bbt} \calm_t$, the set $\cald(A)$ denotes the domain of the linear operator $A : \cald(A) \subset H \to H$ appearing in (\ref{SPDE-manifold}), we use the notation $\beta = A + \alpha$, and $T \calm_t := \frac{d}{dt} \psi(t) + V$ denotes the tangent space to $\calm$ at time $t$; cf., e.g., \cite{Tappe-Wiener}. Denoting by $\partial \bbm := \bbm \cap G$ the boundary of the foliation, we have the decomposition $\bbm = \partial \bbm \oplus \frc \oplus U$, and the tangential condition (\ref{tang-beta}) implies
\begin{align}\label{beta-inc-intro}
\beta_g(v) \in V \quad \text{for all $g \in \partial \bbm$ and all $v \in \frc \oplus U$,}
\end{align}
where we use the notation $\beta_g(v) := \beta(g+v) - \beta(g)$.
As we will see, for every starting point $h_0 \in \bbm$ from the foliation the state process $X$ appearing in (\ref{decomp-intro}) is a solution of the SDE
\begin{align}\label{SDE-X-intro}
\left\{
\begin{array}{rcl}
dX_t & = & \tilde{\beta}(t,X_t)dt + \tilde{\sigma}(t,X_t) dW_t \medskip
\\ X_0 & = & x_0
\end{array}
\right.
\end{align}
for some $x_0 \in \frc \oplus U$, where the coefficients $\tilde{\beta} : \bbs \times \frc \oplus U \to V$ and $\tilde{\sigma} : \bbs \times \frc \oplus U \to V^n$ for an appropriate time interval $\bbs \subset \bbr_+$ are given by
\begin{align}\label{beta-intro}
\tilde{\beta}(t,x) &= \Pi_V \beta(\psi(t_0 + t) + x),
\\ \label{sigma-intro} \tilde{\sigma}(t,x) &= \sigma(\psi(t_0 + t) + x),
\end{align}
for some $t_0 \in \bbr_+$. Here the projection $\Pi_V$ refers to the direct sum decomposition $H = G \oplus V$. From (\ref{SDE-X-intro})--(\ref{sigma-intro}) we see that the state process $X$ in (\ref{decomp-intro}) is affine if and only if                                                          
\begin{itemize}
\item for each $g \in \partial \bbm$ the mapping
\begin{align}\label{map-beta}
v \mapsto \Pi_V \beta(g+v) : \frc \oplus U \to V
\end{align}
is affine, and

\item for each $g \in \partial \bbm$ the mapping 
\begin{align}\label{map-sigma}
v \mapsto \sigma(g+v) : \frc \oplus U \to V^n 
\end{align}
is square-affine,
\end{itemize}
and that the state process $X$ in (\ref{decomp-intro}) is admissible if and only if
\begin{itemize}
\item for each $g \in \partial \bbm$ the mapping (\ref{map-beta}) is inward pointing, and

\item for each $g \in \partial \bbm$ the mapping (\ref{map-sigma}) is parallel.
\end{itemize}
Here the term \emph{square-affine} means that the mapping
\begin{align}\label{sigma-square-intro}
v \mapsto \sigma^2(g+v) := \sigma(g+v) \sigma^*(g+v) : \frc \oplus U \to L(V) 
\end{align}
is affine. In (\ref{sigma-square-intro}) we use the identification $V^n \cong L(\bbr^n,V)$, and concerning the adjoint operator, on $\bbr^n$ we consider the standard inner product, and on $V$ we consider a canonical inner product $\langle \cdot,\cdot \rangle_V$, which is defined by means of the original inner product $\langle \cdot,\cdot \rangle_H$ and the cone $\frc$. Namely, the unique normed basis of the proper cone $\frc$ becomes an orthonormal basis of $C$ under $\langle \cdot,\cdot \rangle_V$, and on $U$ the inner product $\langle \cdot,\cdot \rangle_V$ coincides with the original inner product $\langle \cdot,\cdot \rangle_H$ of the Hilbert space.

Furthermore, the mapping (\ref{map-beta}) is called inward pointing at boundary points of $\frc \oplus U$ (in short \emph{inward pointing}) if
\begin{align*}
\langle \eta,\Pi_V \beta(g+v) \rangle_V \geq 0 \quad \text{for all $v \in \mathfrak{C} \oplus U$ and all $\eta \in \mathfrak{C}$ with $\langle \eta,v \rangle_V = 0$,}
\end{align*}
and the mapping (\ref{map-sigma}) is called parallel to the boundary at boundary points of $\frc \oplus U$ (in short \emph{parallel}) if for each $k=1,\ldots,n$ we have
\begin{align*}
\langle \eta,\sigma_k(g+v) \rangle_V = 0 \quad \text{for all $v \in \mathfrak{C} \oplus U$ and all $\eta \in \mathfrak{C}$ with $\langle \eta,v \rangle_V = 0$.}
\end{align*}
Now, let us present our main result regarding the existence of affine realizations with affine and admissible state processes; we refer to Section~\ref{sec-affine-real} for further details and the precise statements. Recall that, besides the state space $\frc \oplus U$, we fix a set $\fri \subset H$ of initial points. Our essential structural assumption on this set is that it admits a decomposition $\fri = \partial \fri \oplus \frc \oplus U$ with a subset $\partial \fri \subset H$, which we call the boundary of $\fri$, and that $H = G \oplus V$, where $G := \overline{\langle \partial \fri \rangle}$. Conditions (\ref{tang-domain}), (\ref{tang-sigma}), (\ref{beta-inc-intro}) and our explanations concerning the mappings (\ref{map-beta}) and (\ref{map-sigma}) lead us to Theorem~\ref{thm-main-1}, which states that the SPDE (\ref{SPDE-manifold}) has an affine realization with affine and admissible state processes if and only if we have $\mathfrak{I} \subset \cald(A)$ and $\sigma(\overline{\mathfrak{I}}) \subset V^n$, and for each $g \in \partial \mathfrak{I}$ we have
\begin{align}\label{beta-inc-V}
\beta_g(v) &\in V, \quad v \in \mathfrak{C} \oplus U,
\\ \label{beta-affine-inward} v &\mapsto \Pi_V \beta(g + v) : \mathfrak{C} \oplus U \to V \text{ is affine and inward pointing,}
\\ \label{sigma-affine-parallel} v &\mapsto \sigma(g + v) : \mathfrak{C} \oplus U \to V^n \text{ is square-affine and parallel.}
\end{align}
In applications, we often have the situation that the drift is of the form $\alpha = S \sigma^2$ with a linear operator $S \in L(L(V),H)$; in particular, this is case for the mentioned HJMM equation. In this situation, we will derive conditions on the parameters $(A,\sigma)$ and on the set $\fri$ of initial points. The structure $\alpha = S \sigma^2$ of the drift is tailor-made for the existence of affine state processes, provided that the SPDE (\ref{SPDE-manifold}) has an affine realization. Let us briefly outline our main result in this situation; we refer to Section~\ref{sec-drift} for further details and the precise statements. If the mapping $v \mapsto \sigma(g + v)$ in (\ref{sigma-affine-parallel}) is square-affine, then the mapping $v \mapsto \Pi_V \beta(g + v)$ in (\ref{beta-affine-inward}) is affine. However, if the mapping $v \mapsto \sigma(g + v)$ in (\ref{sigma-affine-parallel}) is square-affine and parallel, this does generally not imply that the mapping $v \mapsto \Pi_V \beta(g + v)$ in (\ref{beta-affine-inward}) is affine and inward pointing; as we will show, this is the case if and only if for each $g \in \partial \fri$ we have
\begin{align}\label{cond-AR-1}
\Pi_V(Ag + S \sigma^2(g)) &\in \mathfrak{C} \oplus U,
\\ \label{cond-AR-2} Ac + S \sigma_g^2(c) &\in (\mathfrak{C} + \langle c \rangle) \oplus U, \quad c \in \partial \mathfrak{C},
\\ \label{cond-AR-3} Au &\in U, \quad u \in U,
\end{align}
where $\partial \frc$ denotes the edges of the cone $\frc$.
This leads us to our next result (see Theorem~\ref{thm-main-2}) which states that the SPDE (\ref{SPDE-manifold}) with drift being of the form $\alpha = S \sigma^2$ has an affine realization with affine and admissible state processes if and only if we have $\mathfrak{I} \subset \cald(A)$, and for each $g \in \partial \mathfrak{I}$ we have (\ref{sigma-affine-parallel})--(\ref{cond-AR-3}).

Condition (\ref{beta-inc-V}) from our general result (Theorem~\ref{thm-main-1}) has further consequences in the situation where the drift is of the form $\alpha = S \sigma^2$. In order to outline these consequences, we define the finite dimensional subspace $\calk \subset L(V)$ as
$\calk := S^{-1}(V) \cap R$,
where $R := \langle \sigma^2(\fri) \rangle$, and the finite dimensional subspace $\call \subset L(V,L(V))$ as
$\call := L(V,\calk)$.
Then we have the following results:
\begin{itemize}
\item If the SPDE (\ref{SPDE-manifold}) has an affine realization with affine (but not necessarily admissible) state processes, then the mapping
\begin{align}\label{mapping-const}
g \mapsto \sigma_g^2 : \partial \fri \to L(V,L(V)),
\end{align}
where we use the notation $\sigma_g^2(v) := \sigma^2(g+v) - \sigma^2(g)$,
is constant modulo $\call$; see Proposition~\ref{prop-const-mod}. In particular, if $\calk = \{ 0 \}$, then the mapping (\ref{mapping-const}) must be constant.

\item If the SPDE (\ref{SPDE-manifold}) has an affine realization, and in addition we have $V \cap S(R) = \{ 0 \}$ and ${\rm ker}(S) \cap R = \{ 0 \}$, then the SPDE (\ref{SPDE-manifold}) has an affine realization with affine (but not necessarily admissible) state processes; see Proposition~\ref{prop-Damir}. This result can be regarded as a generalization of \cite[Prop.~9.3]{Filipovic}, which is a result for interest rate models.
\end{itemize}
In Section~\ref{sec-suff}, we will also consider the structure $\alpha = S \sigma^2$ of the drift and provide sufficient conditions on the parameters $(A,\sigma)$ for the existence of an affine realization with affine and admissible state processes, without specifying the set $\fri$ of initial points in advance. Instead of that, our result (Proposition~\ref{prop-suff}) provides a construction of the set of initial points, and we will see that this construction of $\fri$ is the largest possible. We will apply the just described result (Proposition~\ref{prop-suff}) for the construction of the maximal set of initial points for concrete examples of SPDEs like the Hull-White extension of the Cox-Ingersoll-Ross model from interest rate theory.

The remainder of this paper is organized as follows. In Section~\ref{sec-foliations} we provide the required results about invariant foliations for SPDEs. In Section~\ref{sec-affine-real} we examine the existence of affine realizations with affine and admissible state processes. In Section~\ref{sec-drift} we study the situation with the mentioned structure $\alpha = S \sigma^2$ of the drift, and in Section~\ref{sec-suff} we provide sufficient conditions for the existence of affine realizations with affine and admissible state processes, and construct the maximal set of initial points. In Section~\ref{sec-HJMM} we present the HJMM equation and show how it fits into our framework. In Section~\ref{sec-HJMM-examples} we present examples of the HJMM equation with affine realizations and affine and admissible state processes, and construct the maximal sets of initial curves. In Section~\ref{sec-nat-sci} we treat linear SPDEs and present further examples arising from natural sciences. For convenience of the reader, we provide the crucial results about convex cones and affine mappings in Appendix~\ref{app-affine}.

\section{Invariant foliations for SPDEs}\label{sec-foliations}

In this section, we provide the required results about invariant foliations for SPDEs. For further details about SPDEs of the type (\ref{SPDE-manifold}) we refer to \cite{Da_Prato}, \cite{Prevot-Roeckner} or \cite{Atma-book}, and for more details about invariant foliations, we refer to \cite{Tappe-Wiener}. Let $H$ be a separable Hilbert space and let $A : \mathcal{D}(A) \subset H \rightarrow H$ be the infinitesimal generator of a $C_0$-semigroup on $H$. Let $\alpha : H \to H$ and $\sigma : H \to H^n$ (for some positive integer $n \in \bbn$) be continuous mappings.

\begin{remark}\label{remark-solutions}
We call a filtered probability space $\mathbb{B} = (\Omega,\mathcal{F},(\mathcal{F}_t)_{t \in \bbr_+},\mathbb{P})$ satisfying the usual conditions a stochastic basis. In this paper,
the concepts of strong, weak and mild solutions to (\ref{SPDE-manifold}) are understood in a martingale sense (cf. \cite[Chap.~8]{Da_Prato}), that is, we do not fix a stochastic basis $\mathbb{B}$ in advance, but rather call a pair $(r,W)$ -- where $r$ is a continuous, adapted process and $W$ a $\bbr^n$-valued standard Wiener process on some stochastic basis $\mathbb{B}$ -- a strong, weak or mild solution to (\ref{SPDE-manifold}), if the process $r$ has the respective property.
\end{remark}

Let $\frc \subset H$ be a finite dimensional proper convex cone (see Appendix~\ref{app-affine} for further details) and let $U \subset H$ be a finite dimensional subspace such that $C \cap U = \{ 0 \}$, where $C = \langle \frc \rangle$. We assume that the subspace $V = C \oplus U$ satisfies $\dim V \geq 1$. Let $G \subset H$ be a closed subspace such that the Hilbert space admits the direct sum decomposition $H = G \oplus V$.
We introduce the set of intervals
\begin{align*}
\mathbb{J} := \{ [0, \delta] : \delta \in (0, \infty) \} \cup \{ \bbr_+ \}.
\end{align*}
For what follows, we fix an interval $\mathbb{T} \in \mathbb{J}$. For $t_0 \in \mathbb{T}$ we define the interval $\mathbb{T}_{t_0} \in \bbj$ as
\begin{align*}
\mathbb{T}_{t_0} := \{ t \in \mathbb{R}_+ : t_0 + t \in \mathbb{T} \}.
\end{align*}

\begin{definition}
A family $(\mathcal{M}_t)_{t \in \mathbb{T}}$ of subsets
$\mathcal{M}_t \subset H$, $t \in \mathbb{T}$ is called a \emph{foliation}
generated by $\mathfrak{C} \oplus U$, if there exists a mapping $\psi \in C^1(\mathbb{T};G)$ such that
\begin{align}
\mathcal{M}_t = \psi(t) \oplus \mathfrak{C} \oplus U \quad \text{for all $t \in \mathbb{T}$.}
\end{align}
The mapping $\psi$ is called a \emph{parametrization} of the foliation
$(\mathcal{M}_t)_{t \in \mathbb{T}}$.
\end{definition}

In what follows, let $(\mathcal{M}_t)_{t \in \mathbb{T}}$ be a foliation
generated by $\mathfrak{C} \oplus U$.

\begin{remark}
Note that the parametrization $\psi$ of $(\mathcal{M}_t)_{t \in \mathbb{T}}$ is unique, because we demand that it has its values in $G$.
\end{remark}

\begin{definition}
We define the union of all leaves $\bbm := \bigcup_{t \in \bbt} \calm_t$ and the boundary $\partial \bbm := \bbm \cap G$.
\end{definition}

Note that we have the decomposition $\bbm = \partial \bbm \oplus \frc \oplus U$.

\begin{definition}
For each $t \in \bbt$ we define the tangent space $T \calm_t := \frac{d}{dt} \psi(t) \oplus V$.
\end{definition}

\begin{definition}\label{def-inv-foliation}
The foliation $(\mathcal{M}_t)_{t \in \bbt}$ is called \emph{invariant} for the SPDE (\ref{SPDE-manifold}) if for all $t_0 \in
\bbt$ and $h_0 \in \mathcal{M}_{t_0}$ there is a weak solution $r = (r_t)_{t \in \mathbb{T}_{t_0}}$ to (\ref{SPDE-manifold}) with $r_0 = h_0$ such that $r_{\bullet} \in \calm_{t_0 + \bullet}$ up to an evanescent set\footnote[3]{A random set $A \subset \Omega \times \mathbb{R}_+$ is called \emph{evanescent} if the set $\{ \omega \in \Omega : (\omega,t) \in A \text{ for some } t \in \mathbb{R}_+ \}$ is a $\mathbb{P}$-nullset, cf. \cite[1.1.10]{Jacod-Shiryaev}.}.
\end{definition}

For what follows, we define the mapping $\beta := A + \alpha : \mathcal{D}(A) \rightarrow H$.

\begin{proposition}\label{prop-inv-foli-pre}
The following statements are true:
\begin{enumerate}
\item If the foliation $(\mathcal{M}_t)_{t \in \mathbb{T}}$ is invariant in the for (\ref{SPDE-manifold}), then we have (\ref{tang-domain})--(\ref{tang-sigma}).

\item If we have (\ref{tang-domain})--(\ref{tang-sigma}), then we have (\ref{beta-inc-intro}), and $A$ and $\beta$ are continuous on $\bbm$.
\end{enumerate}
\end{proposition}

\begin{proof}
It is obvious that (\ref{tang-domain}) and (\ref{tang-beta}) imply (\ref{beta-inc-intro}). The proof of the remaining assertions is analogous to that of \cite[Thm.~2.11]{Tappe-Wiener}, and therefore omitted.
\end{proof}

For the rest of this section, suppose that these conditions (\ref{tang-domain})--(\ref{tang-sigma}) are fulfilled. The upcoming two definitions correspond to our discussion from Section \ref{sec-intro}. We refer to Appendix~\ref{app-affine} for further details and explanations concerning the following concepts.

\begin{definition}
The foliation $(\calm_t)_{t \in \bbt}$ is called \emph{affine} for the SPDE (\ref{SPDE-manifold}) if for each $g \in \partial \bbm$ the mapping (\ref{map-beta}) is affine and the mapping (\ref{map-sigma}) is square-affine.
\end{definition}

\begin{definition}
The foliation $(\calm_t)_{t \in \bbt}$ is called \emph{affine and admissible} for the SPDE (\ref{SPDE-manifold}) if for each $g \in \partial \bbm$ the mapping (\ref{map-beta}) is affine and inward pointing and the mapping (\ref{map-sigma}) is square-affine and parallel.
\end{definition}

\begin{proposition}\label{prop-SDE}
Suppose that for each $t_0 \in \bbt$ and each $x_0 \in \frc \oplus U$ the SDE (\ref{SDE-X-intro}) has a $\frc \oplus U$-valued strong solution $X = (X_t)_{t \in \bbt_{t_0}}$ (in the sense of Remark~\ref{remark-solutions}), where $\tilde{\beta} : \bbt_{t_0} \times \frc \oplus U \to V$ and $\tilde{\sigma} : \bbt_{t_0} \times \frc \oplus U \to V^n$ are given by (\ref{beta-intro}) and (\ref{sigma-intro}). Then the foliation $(\calm_t)_{t \in \bbt}$ is invariant for (\ref{SPDE-manifold}).
\end{proposition}

\begin{proof}
Let $t_0 \in \bbt$ and $h_0 \in \mathcal{M}_{t_0}$ be arbitrary. Then there exists a unique $x_0 \in \mathfrak{C} \oplus U$ such that $h_0 = \psi(t_0) + x_0$. We define the process $r = (r_t)_{t \in \bbt_{t_0}}$ as $r_t := \psi(t_0 + t) + X_t$, where $X = (X_t)_{t \in \bbt_{t_0}}$ is a $\frc \oplus U$-valued strong solution to (\ref{SDE-X-intro}) with $X_0 = x_0$. Then we have $r_{\bullet} \in \calm_{t_0 + \bullet}$. Now, let $t \in \bbt_{t_0}$ be arbitrary. By (\ref{tang-beta}) we have
\begin{align*}
\Pi_G \beta(r_s) = \frac{d}{ds} \psi(t_0 + s) \quad \text{for all $s \in [0,t]$,}
\end{align*}
and hence
\begin{align*}
r_t &= \psi(t_0) + \big( \psi(t_0 + t) - \psi(t_0) \big) + x_0 + \int_0^t \tilde{\beta}(s,X_s) ds + \int_0^t \tilde{\sigma}(s,X_s) dW_s
\\ &= h_0 + \int_0^t \frac{d}{ds} \psi(t_0 + s) ds + \int_0^t \Pi_V \beta(\psi(t_0 + s) + X_s) ds 
\\ &\quad + \int_0^t \sigma(\psi(t_0 + s) + X_s) dW_s
\\ &= h_0 + \int_0^t \Pi_G \beta(r_s) ds + \int_0^t \Pi_V \beta(r_s) ds + \int_0^t \sigma(r_s)dW_s
\\ &= h_0 + \int_0^t \big( A r_s + \alpha(r_s) \big) ds + \int_0^t \sigma(r_s)dW_s,
\end{align*}
showing that $r$ is a strong solution to (\ref{SPDE-manifold}) with $r_0 = h_0$.
\end{proof}

\begin{proposition}\label{prop-foliation-affine-inv}
If the foliation $(\calm_t)_{t \in \bbt}$ is affine and admissible for (\ref{SPDE-manifold}), then it is also invariant for (\ref{SPDE-manifold}).
\end{proposition}

\begin{proof}
This is a consequence of Proposition~\ref{prop-SDE} and \cite[Thms.~2.13 and 2.14]{Filipovic-inh}.
\end{proof}

\section{Existence of affine realizations with affine and admissible state processes}\label{sec-affine-real}

In this section, we present our main result concerning the existence of affine realizations with affine and admissible state processes. The general mathematical framework is that of Section \ref{sec-foliations}. The only difference is that we do not specify a subspace $G \subset H$ for a direct sum decomposition $H = G \oplus V$ in advance; instead of that, we only specify the parameters $(A,\alpha,\sigma)$ of the SPDE (\ref{SPDE-manifold}) and the state space $\frc \oplus U$. In addition, let $\mathfrak{I} \subset H$ be a nonempty subset, which we call the set of initial points.

\begin{definition}
The SPDE (\ref{SPDE-manifold}) has an affine realization generated by $\mathfrak{C} \oplus U$ with initial points $\mathfrak{I}$  if for each $h_0 \in \mathfrak{I}$ there exist an interval $\bbt \in \bbj$ and a foliation $(\calm_t)_{t \in \bbt}$ generated by $\mathfrak{C} \oplus U$ with $h_0 \in \calm_0$, which is invariant for (\ref{SPDE-manifold}).
\end{definition}

\begin{definition}
The SPDE (\ref{SPDE-manifold}) has an affine realization generated by $\mathfrak{C} \oplus U$ with initial points $\mathfrak{I}$ and with affine state processes if for each $h_0 \in \mathfrak{I}$ there exist an interval $\bbt \in \bbj$ and a foliation $(\calm_t)_{t \in \bbt}$ generated by $\mathfrak{C} \oplus U$ with $h_0 \in \calm_0$, which is invariant and affine for (\ref{SPDE-manifold}).
\end{definition}

\begin{definition}
The SPDE (\ref{SPDE-manifold}) has an affine realization generated by $\mathfrak{C} \oplus U$ with initial points $\mathfrak{I}$ and with affine and admissible state processes if for each $h_0 \in \mathfrak{I}$ there exist an interval $\bbt \in \bbj$ and a foliation $(\calm_t)_{t \in \bbt}$ generated by $\mathfrak{C} \oplus U$ with $h_0 \in \calm_0$, which is invariant, affine and admissible for (\ref{SPDE-manifold}).
\end{definition}

Concerning the set of initial points, we assume that it admits a decomposition $\fri = \partial \fri \oplus \frc \oplus U$ with a subset $\partial \fri \subset H$, which we call the boundary of $\fri$, and that $H = G \oplus V$, where $G := \overline{\langle \partial \fri \rangle}$. In the sequel, we denote by $\Pi_G : H \to G$ and $\Pi_V : H \to V$ the corresponding projections.

\begin{assumption}\label{ass-open}
We suppose that $\partial \fri \cap \cald(A)$ is open in $G \cap \cald(A)$ with respect to the graph norm $\| \cdot \|_{\cald(A)}$, which is given by
\begin{align*}
\| h \|_{\cald(A)} = \sqrt{ \| h \|_H^2 + \| Ah \|_H^2 }, \quad h \in \cald(A).
\end{align*}
\end{assumption}

\begin{assumption}\label{ass-local-Lipschitz}
We suppose that $\alpha : H \to H$ is Lipschitz continuous with respect to $\| \cdot \|_H$, that $\alpha (\cald(A)) \subset \cald(A)$ and that $\alpha|_{\cald(A)} : \cald(A) \to \cald(A)$ is Lipschitz continuous with respect to $\| \cdot \|_{\cald(A)}$.
\end{assumption}

\begin{theorem}\label{thm-main-1}
Suppose that Assumptions~\ref{ass-open} and \ref{ass-local-Lipschitz} are fulfilled. Then the following statements are equivalent:
\begin{enumerate}
\item[(i)] The SPDE (\ref{SPDE-manifold}) has an affine realization generated by $\mathfrak{C} \oplus U$ with initial points $\mathfrak{I}$ and with affine and admissible state processes.

\item[(ii)] We have $\mathfrak{I} \subset \cald(A)$ and $\sigma(\overline{\mathfrak{I}}) \subset V^n$, and for each $g \in \partial \mathfrak{I}$ we have (\ref{beta-inc-V})--(\ref{sigma-affine-parallel}).
\end{enumerate}
\end{theorem}

\begin{proof}
(i) $\Rightarrow$ (ii): This is a consequence of Proposition~\ref{prop-inv-foli-pre}.

\noindent (ii) $\Rightarrow$ (i): Let $h_0 \in \mathfrak{I}$ be arbitrary. Then there are unique $g_0 \in \partial \fri$ and $v_0 \in \frc \oplus U$ such that $h_0 = g_0 + v_0$. Since $\partial \fri$ is open in $G \cap \cald(A)$ with respect to the graph norm $\| \cdot \|_{\cald(A)}$, there exists $\epsilon > 0$ such that
\begin{align*}
B_{\epsilon}(g_0) \subset \partial \fri,
\end{align*}
where $B_{\epsilon}(g_0) \subset G \cap \cald(A)$ denotes the open ball
\begin{align*}
B_{\epsilon}(g_0) = \{ g \in G \cap \cald(A) : \| g - g_0 \|_{\cald(A)} < \epsilon \}.
\end{align*}
According to \cite[Thm.~6.1.7]{Pazy}, there exists a classical solution $\phi \in C^1(\bbr_+;H)$ with $\phi(\bbr_+) \subset \cald(A)$ of the deterministic evolution equation
\begin{align*}
\left\{
\begin{array}{rcl}
\frac{d}{dt} \phi(t) & = & A \phi(t) + \alpha(\phi(t)) \medskip
\\ \phi(0) & = & h_0.
\end{array}
\right.
\end{align*}
Since $\phi : \bbr_+ \to (\cald(A),\| \cdot \|_{\cald(A)})$ is continuous, there exists $\delta > 0$ such that
\begin{align*}
\phi(t) \in B_{\epsilon}(g_0) \oplus V \quad \text{for all $t \in \bbt$,}
\end{align*}
where $\bbt \in \bbj$ denotes the interval $\bbt := [0,\delta]$. Therefore, defining $\psi : \bbt \to G$ as $\psi(t) := \Pi_G \phi(t)$, $t \in \bbt$, we have $\psi(0) = g_0$ and
\begin{align}\label{psi-in-boundary}
\psi(t) \in \partial \fri \quad \text{for all $t \in \bbt$.}
\end{align}
Furthermore, the function $\Pi_V \phi : \bbt \to V$ is a solution to the $V$-valued time-inhomogeneous ODE
\begin{align*}
\left\{
\begin{array}{rcl}
\frac{d}{dt} \varphi(t) & = & \Pi_V \beta(\psi(t) + \varphi(t)) \medskip
\\ \varphi(0) & = & v_0.
\end{array}
\right.
\end{align*}
Therefore, by (\ref{beta-affine-inward}) and (\ref{psi-in-boundary}) we deduce that $\Pi_V \phi(t) \in \frc \oplus U$ for all $t \in \bbt$, and hence
\begin{align}\label{phi-in-I}
\phi(t) \in \fri \quad \text{for all $t \in \bbt$.}
\end{align}
We define the foliation $(\calm_t)_{t \in \bbt}$ as $\calm_t := \psi(t) \oplus \frc \oplus U$. Then we have $h_0 \in \calm_0$ and $\bbm \subset \fri$, and hence, conditions (\ref{tang-domain}) and (\ref{tang-sigma}) are fulfilled. Moreover, by (\ref{psi-in-boundary}), (\ref{phi-in-I}) and (\ref{beta-inc-V}), for all $t \in \bbt$ we obtain
\begin{align*}
\frac{d}{dt} \psi(t) &= \frac{d}{dt} \Pi_G \phi(t) = \Pi_G \frac{d}{dt} \phi(t) = \Pi_G \beta(\phi(t)) 
\\ &= \Pi_G \beta(\psi(t) + \Pi_V \phi(t)) = \Pi_G \big( \underbrace{\beta_{\psi(t)} (\Pi_V \phi(t))}_{\in V} + \beta(\psi(t)) \big) = \Pi_G \beta(\psi(t)),
\end{align*}
and therefore
\begin{align*}
\beta(\psi(t)) \in T \calm_t.
\end{align*}
Thus, by (\ref{psi-in-boundary}) and (\ref{beta-inc-V}), for all $t \in \bbt$ and all $v \in \frc \oplus U$ we deduce that
\begin{align*}
\beta(\psi(t) + v) = \beta(\psi(t)) + \underbrace{\beta_{\psi(t)}(v)}_{\in V} \in T \calm_t,
\end{align*}
showing (\ref{tang-beta}). Furthermore, by virtue of (\ref{beta-affine-inward}) and (\ref{sigma-affine-parallel}) the foliation $(\calm_t)_{t \in \bbt}$ is affine and admissible for (\ref{SPDE-manifold}), and hence, by Proposition~\ref{prop-foliation-affine-inv} it is also invariant for (\ref{SPDE-manifold}).
\end{proof}

\begin{remark}\label{remark-add}
Concerning Theorem~\ref{thm-main-1}, let us make the following additional remarks.
\begin{itemize}
\item Assumptions~\ref{ass-open} and \ref{ass-local-Lipschitz} are only required for the proof of the implication (ii) $\Rightarrow$ (i).

\item If $\sigma$ is additionally Lipschitz continuous, then analogous versions of Theorem~\ref{thm-main-1} concerning the existence of affine realizations and concerning the existence of affine realizations with affine (but not necessarily admissible) state processes hold true. In these situations, conditions (\ref{beta-affine-inward}) and (\ref{sigma-affine-parallel}) can be weakened, and Assumptions~\ref{ass-open} and \ref{ass-local-Lipschitz} and the Lipschitz continuity of $\sigma$ are only required for the proof of the implication (ii) $\Rightarrow$ (i).
\end{itemize}
\end{remark}

\section{SPDEs with drift depending on the volatility}\label{sec-drift}

In this section, we present results concerning the existence of affine realizations with affine and admissible state processes for SPDEs with drift term having a particular structure depending on the volatility. The general mathematical framework is that of Section \ref{sec-affine-real}. In addition to that, we will impose the following assumption which specifies the structure of the drift.

\begin{assumption}\label{ass-factor}
We suppose that the following conditions are fulfilled:
\begin{enumerate}
\item We have $\sigma(H) \subset V^n$.

\item The mapping $\sigma^2 : H \to L(V)$ is Lipschitz continuous.

\item There is a linear operator $S \in L(L(V),H)$ with ${\rm ran}(S) \subset \cald(A)$ such that $\alpha = S \sigma^2$. 
\end{enumerate}
\end{assumption}

In Section \ref{sec-HJMM}, we will see that Assumption~\ref{ass-factor} is in particular satisfied for the HJMM equation from mathematical finance. Since $L(V)$ is finite dimensional, Assumption~\ref{ass-factor} implies that Assumption~\ref{ass-local-Lipschitz} is fulfilled.

\begin{lemma}\label{lemma-inc}
Suppose that $\fri \subset \cald(A)$. Then, for each $g \in \partial \fri$ the following statements are true:
\begin{enumerate}
\item We have
\begin{align*}
\beta_g(v) = Av + S \sigma_g^2(v), \quad v \in \frc \oplus U.
\end{align*}
\item We have (\ref{beta-inc-V}) if and only if
\begin{align*}
Av + S \sigma_g^2(v) \in V, \quad v \in \frc \oplus U.
\end{align*}
\end{enumerate}
\end{lemma}

\begin{proof}
For each $v \in \frc \oplus U$ we have
\begin{align*}
\beta_g(v) = \beta(g+v) - \beta(g) = A(g+v) + S \sigma^2(g+v) - Ag - S \sigma^2(g) = Av + S \sigma_g^2(v),
\end{align*}
which establishes the proof.
\end{proof}

The particular structure $\alpha = S \sigma^2$ implies that $\beta$ is affine, provided that $\sigma$ is square-affine. More precisely, we have the following auxiliary result.

\begin{lemma}\label{lemma-sq-affine-affine}
Suppose that $\fri \subset \cald(A)$, and let $g \in \partial \fri$ be such that the mapping $v \mapsto \sigma(g+v)$ in (\ref{sigma-affine-parallel}) is square-affine. Then the mapping $v \mapsto \Pi_V \beta(g+v)$ in (\ref{beta-affine-inward}) is affine.
\end{lemma}

\begin{proof}
This is a direct consequence of the structure $\beta = A + S \sigma^2$.
\end{proof}

However, if $\sigma$ is additionally parallel, this does generally not imply that $\beta$ is inward pointing; here is a criterion.

\begin{lemma}\label{lemma-beta-admissible}
Suppose that $\fri \subset \cald(A)$, and let $g \in \partial \mathfrak{I}$ be such that condition (\ref{sigma-affine-parallel}) is fulfilled. Then the following statements are equivalent:
\begin{enumerate}
\item[(i)] We have (\ref{beta-inc-V}) and (\ref{beta-affine-inward}).

\item[(ii)] We have (\ref{cond-AR-1})--(\ref{cond-AR-3}).
\end{enumerate}
\end{lemma}

\begin{proof}
By Proposition~\ref{prop-char-affine-2} we have 
\begin{align}\label{sigma-2-U-zero}
\sigma_g^2(u) = 0, \quad u \in U.
\end{align}
By virtue of (\ref{sigma-2-U-zero}), conditions (\ref{cond-AR-2}) and (\ref{cond-AR-3}) imply (\ref{beta-inc-V}). 
Now, suppose that condition (\ref{beta-inc-V}) is fulfilled. We define $\beta_1 \in V$ and $\beta_2 \in L(V)$ as $\beta_1 := \Pi_V (Ag + S \sigma^2(g))$ and $\beta_2(v) := Av + S \sigma_g^2(v)$. Then, by (\ref{beta-inc-V}), for each $v \in \frc \oplus U$ we have
\begin{align*}
\Pi_V \beta(g+v) = \Pi_V \big( \beta(g) + \beta_g(v) \big) = \Pi_V \beta(g) + \beta_g(v) = \beta_1 + \beta_2(v).
\end{align*}
Therefore, by Proposition~\ref{prop-char-affine-1} and (\ref{sigma-2-U-zero}), condition (\ref{beta-affine-inward}) is equivalent to (\ref{cond-AR-1})--(\ref{cond-AR-3}).
\end{proof}

\begin{theorem}\label{thm-main-2}
Suppose that Assumptions~\ref{ass-open} and \ref{ass-factor} are fulfilled. Then the following statements are equivalent:
\begin{enumerate}
\item[(i)] The SPDE (\ref{SPDE-manifold}) has an affine realization generated by $\mathfrak{C} \oplus U$ with initial points $\mathfrak{I}$ and with affine and admissible state processes.

\item[(ii)] We have $\mathfrak{I} \subset \cald(A)$, and for each $g \in \partial \mathfrak{I}$ we have (\ref{sigma-affine-parallel})--(\ref{cond-AR-3}).
\end{enumerate}
\end{theorem}

\begin{proof}
This is a consequence of Theorem~\ref{thm-main-1} and Lemma~\ref{lemma-beta-admissible}.
\end{proof}

The condition (\ref{beta-inc-V}) from our general result (Theorem~\ref{thm-main-1}) has further consequences in the present situation where the drift is of the form $\alpha = S \sigma^2$. In order to outline these consequences, we define the finite dimensional subspace $\calk \subset L(V)$ as $\calk := S^{-1}(V) \cap R$, where $R := \langle \sigma^2(\fri) \rangle$, and the finite dimensional subspace $\call \subset L(V,L(V))$ as
$\call := L(V,\calk)$.

\begin{proposition}\label{prop-const-mod}
Suppose that $\fri \subset \cald(A)$ and that for each $g \in \partial \fri$ condition (\ref{beta-inc-V}) is fulfilled. Then the following statements are true:
\begin{enumerate}
\item For each $v \in \frc \oplus U$ the mapping
\begin{align}\label{const-map-1}
g \mapsto \sigma_g^2(v) : \partial \fri \to L(V)
\end{align}
is constant modulo $\calk$.

\item If for each $g \in \partial \fri$ the mapping in (\ref{sigma-affine-parallel}) is square-affine, then the mapping
\begin{align}\label{const-map-2}
g \mapsto \sigma_g^2 : \partial \mathfrak{I} \to L(V,L(V))
\end{align}
is constant modulo $\call$.
\end{enumerate}
\end{proposition}

\begin{proof}
Let $v \in \frc \oplus U$ be arbitrary. Furthermore, let $g_1,g_2 \in \partial \fri$ be arbitrary. By Lemma~\ref{lemma-inc} we have
\begin{align*}
S \big( \sigma_{g_1}^2(v) - \sigma_{g_2}^2(v) \big) = \big( Av + S \sigma_{g_1}^2(v) \big) - \big( Av + S \sigma_{g_2}^2(v) \big) \in V,
\end{align*}
which implies
\begin{align*}
\sigma_{g_1}^2(v) - \sigma_{g_2}^2(v) \in \calk.
\end{align*}
This proves the first statement, and the second statement is an immediate consequence.
\end{proof}

In particular, if the SPDE (\ref{SPDE-manifold}) has an affine realization and we have $\calk = \{ 0 \}$, then the mapping (\ref{const-map-1}), or (\ref{const-map-2}), respectively, must be constant. The following result can be regarded as a generalization of \cite[Prop.~9.3]{Filipovic}, which is a result for interest rate models.

\begin{proposition}\label{prop-Damir}
Suppose that $V \cap S(R) = \{ 0 \}$ and ${\rm ker}(S) \cap R = \{ 0 \}$, and that the SPDE (\ref{SPDE-manifold}) has an affine realization generated by $\mathfrak{C} \oplus U$ with initial points $\mathfrak{I}$. Then the SPDE (\ref{SPDE-manifold}) has an affine realization generated by $\mathfrak{C} \oplus U$ with initial points $\mathfrak{I}$ and with affine (but not necessarily admissible) state processes.
\end{proposition}

\begin{proof}
By Remark~\ref{remark-add} we have $\mathfrak{I} \subset \cald(A)$ and (\ref{beta-inc-V}). Let $g \in \partial \mathfrak{I}$ be arbitrary. By Lemma~\ref{lemma-inc} we have
\begin{align*}
Av = \beta_g(v) - S \sigma_g^2(v), \quad v \in \frc \oplus U.
\end{align*}
Since $V \cap S(R) = \{ 0 \}$ we obtain that $\beta_g \in L(V)$ and $S \sigma_g^2 \in L(V,S(R))$. Since ${\rm ker}(S) \cap R = \{ 0 \}$ we deduce that $\sigma_g^2 \in L(V,L(V))$. Consequently, the mapping $v \mapsto \sigma(g+v)$ in (\ref{sigma-affine-parallel}) is square-affine. Therefore, by Lemma~\ref{lemma-sq-affine-affine} the mapping $v \mapsto \Pi_V \beta(g+v)$ in (\ref{beta-affine-inward}) if affine, which completes the proof.
\end{proof}

Now, we derive some consequences regarding the existence of affine realizations generated by the subspace $V$; that is, now, there is no proper cone contained in the structure of the state space. For this purpose, we will require the concept of quasi-exponential volatilities.

\begin{definition}\label{def-qe}
We introduce the following notions:
\begin{enumerate}
\item If $\sigma_k(H) \subset \cald(A^{\infty})$ for all $k=1,\ldots,n$, then we define the subspace $A_{\sigma} \subset H$ as
\begin{align*}
A_{\sigma} := \sum_{k=1}^n \langle A^m \sigma_k(h) : m \in \bbn_0 \text{ and } h \in H \rangle.
\end{align*}
\item The volatility $\sigma$ is called \emph{$A$-quasi exponential}, if we have $\sigma_k(H) \subset \cald(A^{\infty})$ for all $k=1,\ldots,n$ and $\dim A_{\sigma} < \infty$.
\end{enumerate}
\end{definition}

For some SPDEs (like the HJMM equation in Section \ref{sec-HJMM}) a sufficient condition for the existence of an affine realization is that the volatility $\sigma$ is $A$-quasi-exponential. The following two results provide further conditions on $\sigma$ which are necessary and sufficient in order to obtain affine state processes.

\begin{proposition}\label{prop-qe}
Suppose that the volatility $\sigma$ is $A$-quasi-exponential. Then the following statements are equivalent:
\begin{enumerate}
\item[(i)] The SPDE (\ref{SPDE-manifold}) has an affine realization generated by $V$ with initial points $\cald(A)$ and with affine and admissible state processes.

\item[(ii)] The SPDE (\ref{SPDE-manifold}) has an affine realization generated by $V$ with initial points $\cald(A)$ and with affine state processes.

\item[(iii)] We have $A_{\sigma} \subset V$, and for each $h \in H$ the mapping
\begin{align}\label{map-qe-1}
v \mapsto \sigma^2(h+v) : V \to L(V,L(V))
\end{align}
is constant.
\end{enumerate}
If the previous conditions are fulfilled, then the SPDE (\ref{SPDE-manifold}) has an affine realization generated by $A_{\sigma}$ with initial points $\cald(A)$ and with affine and admissible state processes.
\end{proposition}

\begin{proof}
(i) $\Leftrightarrow$ (ii): This implication is obvious, because $V$ is a linear space.

\noindent (i) $\Rightarrow$ (iii): By Theorem~\ref{thm-main-2} we have $\sigma(H) \subset V^n$ and $A(V) \subset V$, which shows $A_{\sigma} \subset V$. Furthermore, by Remark~\ref{remark-square-affine}, for each $h \in H$ the mapping (\ref{map-qe-1}) is constant.

\noindent (iii) $\Rightarrow$ (i): According to Lemma~\ref{lemma-extend-space} and Theorem~\ref{thm-main-2}, the SPDE (\ref{SPDE-manifold}) has an affine realization generated by $A_{\sigma}$ with initial points $\cald(A)$ and with affine and admissible state processes.
\end{proof}

\begin{corollary}\label{cor-U-2}
Suppose that the volatility $\sigma$ is $A$-quasi-exponential, and that
\begin{align}\label{U-2-cond-1}
\dim \langle \sigma_k(H) \rangle &\leq 1 \quad \text{for all $k = 1,\ldots,n$} \quad \text{and}
\\ \label{U-2-cond-2} \langle \sigma_k(H) \rangle \cap \langle \sigma_l(H) \rangle &= \{ 0 \} \quad \text{for all $k,l = 1,\ldots,n$ with $k \neq l$.}
\end{align}
Then the following statements are equivalent:
\begin{enumerate}
\item[(i)] The SPDE (\ref{SPDE-manifold}) has an affine realization generated by $V$ with initial points $\cald(A)$ and with affine and admissible state processes.

\item[(ii)] The SPDE (\ref{SPDE-manifold}) has an affine realization generated by $V$ with initial points $\cald(A)$ and with affine state processes.

\item[(iii)] We have $A_{\sigma} \subset V$, and for each $h \in H$ the mapping 
\begin{align*}
v \mapsto \sigma(h + v) : V \to V^n 
\end{align*}
is constant.
\end{enumerate}
If the previous conditions are fulfilled, then the SPDE (\ref{SPDE-manifold}) has an affine realization generated by $A_{\sigma}$ with initial points $\cald(A)$ and with affine and admissible state processes.
\end{corollary}

\begin{proof}
This is an immediate consequence of Propositions~\ref{prop-qe} and \ref{prop-sigma-constant}.
\end{proof}

\section{Sufficient conditions for the existence of affine realizations and construction of the maximal set of initial points}\label{sec-suff}

In this section, we present sufficient conditions for the existence of affine realizations with affine and admissible state processes. The general mathematical framework is that of Section \ref{sec-drift} (in particular we fix a state space of the type $\frc \oplus U$ and the drift is of the form $\alpha = S \sigma^2$), but we do not specify the set $\fri$ of initial points in advance. Instead of that, let $G \subset H$ be a closed subspace such that $H = G \oplus V$.

\begin{proposition}\label{prop-suff}
Suppose that Assumption~\ref{ass-factor} is fulfilled, that for each $g \in G$ we have (\ref{sigma-affine-parallel}), the mapping
\begin{align}\label{map-G-const}
g \mapsto \sigma_g^2 : G \to L(V,L(V))
\end{align}
is constant, and we have
\begin{align}\label{Ac-new}
Ac + S ( \sigma^2(c) - \sigma^2(0) ) &\in (\frc \oplus \langle c \rangle) \oplus U, \quad c \in \partial \frc
\end{align}
and (\ref{cond-AR-3}). Then the SPDE (\ref{SPDE-manifold}) has an affine realization generated by $\frc \oplus U$ with initial points
\begin{align}\label{def-I}
\mathfrak{I} = \{ h \in (G \cap \cald(A)) \oplus \mathfrak{C} \oplus U : \Pi_V(A \Pi_G h + S \sigma^2(\Pi_G h)) \in {\rm Int} \, \mathfrak{C} \oplus U \}
\end{align}
and with affine and admissible state processes, and the set of initial points has the decomposition $\fri = \partial \fri \oplus \frc \oplus U$, where the boundary is given by
\begin{align}\label{def-I-boundary}
\partial \mathfrak{I} = \{ g \in G \cap \cald(A) : \Pi_V(Ag + S \sigma^2(g) ) \in {\rm Int} \, \mathfrak{C} \oplus U \}.
\end{align}
\end{proposition}

\begin{proof}
Inspecting the definitions (\ref{def-I}) and (\ref{def-I-boundary}), we see that we have the decomposition $\fri = \partial \fri \oplus \frc \oplus U$, and by the definition (\ref{def-I}) of $\fri$ we see that $\fri \subset \cald(A)$. Noting that $\Pi_V \beta : (G \cap \cald(A), \| \cdot \|_{\cald(A)}) \to (V,\| \cdot \|_V)$ is continuous, that $\partial \fri = (\Pi_V \beta)^{-1}({\rm Int} \, \frc \oplus U)$ by the definition (\ref{def-I-boundary}), and that ${\rm Int} \, \frc \oplus U$ is open in $V$, we deduce that $\partial \fri$ is open in $G \cap \cald(A)$. Therefore, Assumption~\ref{ass-open} is fulfilled, and we also have $G = \overline{\langle \partial \fri \rangle}$, where the closure is taken with respect to the norm on $H$. Furthermore, by the definition (\ref{def-I-boundary}), condition (\ref{cond-AR-1}) is fulfilled for each $g \in \partial \fri$. Since the mapping (\ref{map-G-const}) is constant, for each $g \in \partial \fri$ and each $c \in \partial \frc$ by (\ref{Ac-new}) we obtain
\begin{align*}
Ac + S \sigma_g^2(c) = Ac + S \sigma_0^2(c) = Ac + S (\sigma^2(c) - \sigma^2(0)) \in (\frc \oplus \langle c \rangle) \oplus U,
\end{align*}
showing (\ref{cond-AR-2}). Consequently, by Theorem~\ref{thm-main-2} the SPDE (\ref{SPDE-manifold}) has an affine realization generated by $\frc \oplus U$ with initial points $\fri$ and with affine and admissible state processes.
\end{proof}

\begin{remark}
The condition that the mapping (\ref{map-G-const}) is constant comes from Proposition~\ref{prop-const-mod}.
\end{remark}

\begin{remark}
Inspecting the proof of Proposition~\ref{prop-suff}, we see that the set $\fri \subset \cald(A)$ given by (\ref{def-I}) is the maximal set of initial points such that $\partial \fri$ is open in $G \cap \cald(A)$ with respect to the graph norm $\| \cdot \|_{\cald(A)}$ (see Assumption~\ref{ass-open}) and condition (\ref{cond-AR-1}) is fulfilled.
\end{remark}

We will illustrate Proposition~\ref{prop-suff} in Section \ref{sec-HJMM-examples}, where we present examples of the HJMM equation and construct the maximal sets of initial points.

\section{The HJMM equation}\label{sec-HJMM}

In this section, we apply our results from the previous sections to the HJMM (Heath-Jarrow-Morton-Musiela) equation. This is a SPDE which models the term structure of interest rates in a market of zero coupon bonds.

Let us briefly introduce the model we consider. A zero coupon bond with maturity $T$ is a financial asset that pays the holder one monetary unit at $T$. Its price at $t \leq T$ can be written as the continuous discounting of one unit of the domestic currency
\begin{align*}
P(t,T) = \exp \bigg( -\int_t^T f(t,s) ds \bigg),
\end{align*}
where $f(t,T)$ is the rate prevailing at time $t$ for instantaneous borrowing at time $T$, also called the forward rate for date $T$.

After transforming the original HJM (Heath-Jarrow-Morton) dynamics of the forward rates (see \cite{HJM}) by means of the Musiela parametrization
$r_t(x) = f(t,t+x)$ (see \cite{Musiela}), the forward rates can be considered as a weak solution to the HJMM (Heath-Jarrow-Morton-Musiela) equation
\begin{align}\label{HJMM}
\left\{
\begin{array}{rcl}
dr_t & = & \big( \frac{d}{dx} r_t + \alpha_{\rm HJM}(r_t) \big) dt +
\sigma(r_{t})dW_t
\medskip
\\ r_0 & = & h_0,
\end{array}
\right.
\end{align}
which is a particular SPDE of the type (\ref{SPDE-manifold}). The state space of the HJMM equation (\ref{HJMM}) is a separable Hilbert space $H$ of forward curves $h : \bbr_+ \to \bbr$, and $d/dx$ denotes the differential operator, which is generated by the translation semigroup. In order to ensure absence of arbitrage in the bond market, we consider the HJMM equation (\ref{HJMM}) under a martingale measure. Then the drift term $\alpha_{\rm HJM} : H \to H$ is given by
\begin{align}\label{HJM-drift}
\alpha_{\rm HJM}(h) = \sum_{k=1}^n \sigma_k(h) \Sigma_k(h), \quad h \in H,
\end{align}
where $\Sigma = (\Sigma_1,\ldots,\Sigma_n) : H \to H^n$ is defined as
\begin{align*}
\Sigma_k(h) := \int_0^{\bullet} \sigma_k(h)(\eta) d\eta \quad \text{for $h \in H$ and $k=1,\ldots,n$.} 
\end{align*}
We refer, e.g., to \cite{fillnm} for further details concerning the derivation of (\ref{HJMM}) and the drift condition (\ref{HJM-drift}). Furthermore, the following choice of the state space, which has been utilized in \cite{fillnm}, has all properties which we require in the sequel. We fix a nondecreasing $C^1$-function $w : \bbr_+ \to [1,\infty)$ such that $w^{-1/3} \in \mathcal{L}^1(\bbr_+)$, and denote by $H$ the space of all absolutely continuous functions $h : \mathbb{R}_+
\rightarrow \mathbb{R}$ such that
\begin{align*}
\| h \|_H := \bigg( |h(0)|^2 + \int_{\mathbb{R}_+} |h'(x)|^2
w(x) dx \bigg)^{1/2} < \infty.
\end{align*}
Apart from this particular choice of the state space $H$ and the drift (\ref{HJM-drift}), the general mathematical framework is that of Section \ref{sec-affine-real}.

\begin{assumption}\label{ass-factor-HJMM}
We suppose that the following conditions are fulfilled:
\begin{enumerate}
\item We have $V \subset \cald(d/dx)$.

\item We have $\sigma(H) \subset V^n$.

\item The mapping $\sigma^2 : H \to L(V)$ is Lipschitz continuous.
\end{enumerate}
\end{assumption}

The following result shows that Assumption~\ref{ass-factor} is fulfilled, which implies that the HJMM equation (\ref{HJMM}) belongs to the framework considered in the previous sections.

\begin{proposition}\label{prop-HJM-linear}
There is a linear operator $S \in L(L(V),H)$ with ${\rm ran}(S) \subset \cald(d/dx)$ such that $\alpha_{\rm HJM} = S \sigma^2$.
\end{proposition}

\begin{proof}
Let $\lambda = (\lambda_1,\ldots,\lambda_d)$ be a basis of $V$ such that $(\lambda_1,\ldots,\lambda_m)$ is a basis of $\frc$, where $m = \dim C$. As pointed out in Remark~\ref{remark-orth-basis}, we may assume, without loss of generality, that $\lambda$ is an orthonormal basis of $V$. We define $S_{\lambda} \in L(\bbr^{d \times d},H)$ as
\begin{align*}
S_{\lambda} \Phi = \langle \Phi \cdot \lambda, \Lambda \rangle, \quad \Phi \in \bbr^{d \times d}.
\end{align*}
Here $\Phi \cdot \lambda \in H^d$ is understood in the sense of matrix multiplication, the vector $\Lambda = (\Lambda_1,\ldots,\Lambda_d) \in H^d$ is given by the primitives $\Lambda_i := \int_0^{\bullet} \lambda_i(\eta) d\eta$ for $i=1,\ldots,d$, and we use the notation
\begin{align*}
\langle h,g \rangle := \sum_{i=1}^d h_i g_i \quad \text{for $h,g \in H^d$.}
\end{align*}
Since $V \subset \cald(d/dx)$, we have ${\rm ran}(S_{\lambda}) \subset \cald(d/dx)$, due to the properties of the state space $H$. Let $\Psi^{\lambda,\lambda} : L(V) \to \bbr^{d \times d}$ be the canonical isomorphism from Definition~\ref{def-matrix-of-operator}. We define $S \in L(L(V),H)$ as $S := S_{\lambda} \Psi^{\lambda,\lambda}$. Then, by (\ref{HJM-drift}) and Lemma~\ref{lemma-matrices} we obtain
\begin{align*}
\alpha_{\rm HJM} &= \langle \sigma,\Sigma \rangle = \langle \sigma^{\lambda} \cdot \lambda, \sigma^{\lambda} \cdot \Lambda \rangle = \langle (\sigma^{\lambda})^{\top} \cdot \sigma^{\lambda} \cdot \lambda,\Lambda \rangle 
\\ &= S_{\lambda} ((\sigma^{\lambda})^{\top} \cdot \sigma^{\lambda}) = S_{\lambda} (\sigma^2)^{\lambda,\lambda} = S_{\lambda} \Psi^{\lambda,\lambda} (\Psi^{\lambda,\lambda})^{-1} (\sigma^2)^{\lambda,\lambda} = S \sigma^2,
\end{align*}
and we have ${\rm ran}(S) \subset \cald(d/dx)$, because ${\rm ran}(S_{\lambda}) \subset \cald(d/dx)$, completing the proof.
\end{proof}

If the volatility $\sigma$ is Lipschitz continuous and $(d/dx)$-quasi-exponential, then the HJMM equation (\ref{HJMM}) has an affine realization generated by a subspace; see, for example \cite[Prop.~6.4]{Bj_Sv}, \cite[Prop.~6.2]{Tappe-Wiener} or \cite[Prop.~6.2]{Tappe-affin-real}. The corresponding state processes are not necessarily affine processes; Proposition~\ref{prop-qe} and Corollary \ref{cor-U-2} provide criteria on the volatility $\sigma$.

\begin{example}\label{ex-Damir}
Suppose that the volatility $\sigma : H \to H$ is of the form
\begin{align}\label{DDV}
\sigma(h) = \Phi(h) \lambda
\end{align}
with a continuous mapping $\Phi : H \to \bbr$ and $\lambda(x) = e^{-\gamma x}$, $x \in \bbr_+$ for some constant $\gamma \in (0,\infty)$. It is well-known (see, for example \cite{Tappe-Wiener}) that the HJMM equation (\ref{HJMM}) has an affine realization generated by the subspace $V = \langle \lambda,\lambda^2 \rangle$, but the state processes are generally not affine. This does not contradict Proposition~\ref{prop-Damir}; since 
\begin{align*}
S \sigma^2(h) = \alpha_{\rm HJM}(h) = \Phi^2(h) \lambda \Lambda = \Phi^2(h) \frac{\lambda - \lambda^2}{\gamma},
\end{align*}
we have $S(R) = \langle \lambda - \lambda^2 \rangle$, and hence $V \cap S(R) \neq \{ 0 \}$.
\end{example}

For the rest of this section, we present some consequences concerning the existence of one-dimensional realizations. For this purpose, we assume $\dim V = 1$, and that the volatility $\sigma : H \to H$ is of the form (\ref{DDV}) with a continuous mapping $\Phi : H \to \bbr$ and a function $\lambda \in V$. We distinguish between the two cases $\dim U = 1$ and $\dim \frc = 1$, where we recall that $V = C \oplus U$ and $C = \langle \frc \rangle$. First, we assume that $\dim U = 1$. The following consequence complements results about the existence of affine realizations for the Hull-White extension of the Vasi\u{c}ek model; see, for example \cite[Prop.~7.2]{Bj_Sv}.

\begin{proposition}
The following statements are equivalent:
\begin{enumerate}
\item[(i)] The HJMM equation (\ref{HJMM}) has an affine realization generated by $U$ with initial curves\footnote[4]{In the context of the HJMM equation, we agree to speak about initial curves instead of initial points.} $\cald(d/dx)$.

\item[(ii)] The HJMM equation (\ref{HJMM}) has an affine realization generated by $U$ with initial curves $\cald(d/dx)$ and with affine and admissible state processes.

\item[(iii)] There are constants $\rho,\gamma \in \bbr$ such that
\begin{align}\label{exponential-form}
\lambda(x) = \rho \cdot e^{-\gamma x}, \quad x \in \bbr_+,
\end{align}
and for each $h \in H$ the mapping $u \mapsto \Phi(h + u) : U \to \bbr$ is constant.
\end{enumerate}
\end{proposition}

\begin{proof}
(i) $\Rightarrow$ (ii): This implication follows from Proposition~\ref{prop-Damir}, and since $U$ is a linear space.

\noindent (ii) $\Rightarrow$ (i): This implication is obvious.

\noindent (ii) $\Leftrightarrow$ (iii): This equivalence is a consequence of Theorem~\ref{thm-main-2} and Corollary \ref{cor-U-2}.
\end{proof}

Now, suppose that $\dim \frc = 1$, and let $\fri \subset \cald(d/dx)$ be a set of initial curves of the form $\fri = \partial \fri \oplus \frc$ such that $H = G \oplus C$, where $G := \overline{\langle \partial \fri \rangle}$. Without loss of generality, we assume that $\lambda \in \frc$. The following two consequences complement results about the existence of affine realizations for the Hull-White extension of the Cox-Ingersoll-Ross model; see, for example \cite[Prop.~7.3]{Bj_Sv}.

\begin{proposition}
Suppose that the HJMM equation (\ref{HJMM}) has an affine realization generated by $\frc$ with initial curves $\mathfrak{I}$. Then the HJMM equation (\ref{HJMM}) has an affine realization generated by $\frc$ with initial curves $\mathfrak{I}$ and with affine state processes, and there are a mapping $\Psi : H \to \bbr$, a continuous linear functional $\ell \in H^*$ with $\ell(\lambda) = 1$ and $G \subset \ker(\ell)$, and constants $\rho > 0$ and $\gamma \in \bbr$ such that
\begin{align}\label{Psi-zero}
&\Phi(h) = \rho \sqrt{\Psi(\Pi_G h) + \ell(h)}, \quad h \in \fri,
\\ \label{Riccati-pre} &\frac{d}{dx} \lambda + \rho^2 \lambda \Lambda + \gamma \lambda = 0.
\end{align}
\end{proposition}

\begin{proof}
This is a consequence of Proposition~\ref{prop-Damir}, Remark~\ref{remark-add} and Proposition~\ref{prop-const-mod}.
\end{proof}

\begin{proposition}
Suppose that the HJMM equation (\ref{HJMM}) has an affine realization generated by $\frc$ with initial curves $\mathfrak{I}$. Then the following statements are equivalent:
\begin{enumerate}
\item[(i)] Then the HJMM equation (\ref{HJMM}) has an affine realization generated by $\frc$ with initial curves $\mathfrak{I}$ and with affine and admissible state processes.

\item[(ii)] In (\ref{Psi-zero}) we have $\Psi \equiv 0$, and we have $\Pi_C \frac{d}{dx} g \in \langle \lambda \rangle^+$ for all $g \in \partial \fri$.
\end{enumerate}
\end{proposition}

\begin{proof}
This is a consequence of Theorem~\ref{thm-main-2}.
\end{proof}

\section{Examples of the HJMM equation and the maximal set of initial curves}\label{sec-HJMM-examples}

In this section, we present examples of the HJMM equation (\ref{HJMM}) with affine realizations and affine and admissible state processes, and for these examples we construct the maximal set of initial curves. Let $H$ be the state space presented in the previous Section \ref{sec-HJMM}. Throughout this section, we assume that the volatility $\sigma : H \to H$ is of the form
\begin{align}\label{lambda-ell}
\sigma(h) = \rho \sqrt{|\ell(h)|} \lambda
\end{align}
with a function $\lambda \in \cald(d/dx)$, a constant $\rho > 0$ and a continuous linear functional $\ell \in H^*$ such that $\ell(\lambda) = 1$. In our first example, let $\lambda$ be a solution of the Riccati equation (\ref{Riccati-pre}) for some constant $\gamma \in \bbr$.

\begin{remark}
The solution of the Riccati differential equation
\begin{align*}
\frac{d}{dx} \Lambda + \frac{\rho^2}{2} \Lambda^2 + \gamma \Lambda = 1, \quad \Lambda(0) = 0
\end{align*}
is given by
\begin{align}\label{sol-lambda-pre}
\Lambda(x) = \frac{2 (\exp(x \sqrt{\gamma^2 + \rho^2}) - 1)}{(\sqrt{\gamma^2 + 2 \rho^2} + \gamma) (\exp(x \sqrt{\gamma^2 + 2 \rho^2}) - 1) + 2 \sqrt{\gamma^2 + 2 \rho^2}}, \quad x \in \bbr_+,
\end{align}
see, for example, \cite[Sec.~7.4.1]{fillnm}. Therefore, the function
\begin{align}\label{sol-lambda}
\lambda = 1 - \frac{\rho^2}{2} \Lambda^2 - \gamma \Lambda,
\end{align}
where $\Lambda$ is given by (\ref{sol-lambda-pre}), is a solution to the ordinary differential equation (\ref{Riccati-pre}).
\end{remark}

\begin{proposition}\label{prop-CIR}
The HJMM equation (\ref{HJMM}) has an affine realization generated by $\frc = \langle \lambda \rangle^+$ with initial curves
\begin{align*}
\mathfrak{I} = \{ h \in \mathcal{D}(d/dx) : \ell(h) \geq 0 \text{ and } \ell(h') + (\rho^2 \ell(\lambda \Lambda) + \gamma) \ell(h) > 0 \}
\end{align*}
and with affine and admissible state processes, and the set of initial curves has the decomposition $\fri = \partial \fri \oplus \frc$, where the boundary is given by
\begin{align*}
\partial \mathfrak{I} = \{ h \in \mathcal{D}(d/dx) : \ell(h) = 0 \text{ and } \ell(h') > 0 \}.
\end{align*}
\end{proposition}

\begin{proof}
Setting $G := \ker(\ell)$, we have the direct sum decomposition $H = G \oplus C$, the corresponding projections are given by
\begin{align}\label{projections-1}
\Pi_G h = h - \ell(h) \lambda \quad \text{and} \quad \Pi_C h = \ell(h) \lambda \quad \text{for $h \in H$,}
\end{align}
and we have
\begin{align}\label{G-oplus-C}
G \oplus \frc = \{ h \in H : \ell(h) \geq 0 \}.
\end{align}
For each $g \in G$ we have (\ref{sigma-affine-parallel}) and the mapping (\ref{map-G-const}) is constant, and condition (\ref{Ac-new}) is satisfied due to the Riccati equation (\ref{Riccati-pre}). By (\ref{projections-1}) and the Riccati equation (\ref{Riccati-pre}), for each $h \in \cald(d/dx)$ we have
\begin{align*}
\Pi_C \frac{d}{dx} \Pi_G h &= \Pi_C \frac{d}{dx} ( h - \ell(h) \lambda ) = \Pi_C \big( h' + \ell(h) \big(  \rho^2 \lambda \Lambda + \gamma \lambda \big) \big)
\\ &= \big( \ell(h') + \ell(h) \big( \rho^2 \ell(\lambda \Lambda) + \gamma \ell(\lambda) \big) \big) \lambda = \big( \ell(h') + ( \rho^2 \ell(\lambda \Lambda) + \gamma ) \ell(h) \big) \lambda,
\end{align*}
and since $G = \ker(\ell)$, for each $g \in G$ we have
\begin{align*}
S \sigma^2(g) = \alpha_{\rm HJM}(g) = \rho^2 \ell(g) \lambda \Lambda = 0.
\end{align*}
Therefore, and taking into account (\ref{G-oplus-C}), applying Proposition~\ref{prop-suff} completes the proof.
\end{proof}

\begin{remark}\label{rem-CIR}
Let $h \in \fri$ be arbitrary. By Proposition~\ref{prop-CIR} there exist an interval $\bbt \in \bbj$, a parametrization $\psi \in C^1(\bbt;G)$, and an affine and admissible process $X$ with state space $\bbr_+$ such that the strong solution $r$ to the HJMM equation (\ref{HJMM}) with $r_0 = h$ is given by
\begin{align*}
r_t = \psi(t) + X_t \cdot \lambda, \quad t \in \bbt.
\end{align*}
Applying the functional $\ell$, this gives us
\begin{align*}
\ell(r_t) = \ell(\psi(t) + X_t \cdot \lambda) = X_t, \quad t \in \bbt,
\end{align*}
showing that $\ell(r) = X$. Therefore, $\ell(r)$ is an affine process with state space $\bbr_+$, which acts as state process of the realization.
\end{remark}

A popular choice for the linear functional $\ell \in H^*$ is the evaluation at the short end, that is $\ell(h) = h(0)$. Note that the condition $\ell(\lambda) = 1$ is fulfilled, because $\Lambda(0) = 0$ and we have the representation (\ref{sol-lambda}) of $\lambda$. We obtain the following result.

\begin{corollary}
The HJMM equation (\ref{HJMM}) has an affine realization generated by $\frc = \langle \lambda \rangle^+$ with initial curves
\begin{align}\label{initial-ex-1b}
\mathfrak{I} = \{ h \in \mathcal{D}(d/dx) : h(0) \geq 0 \text{ and } h'(0) + \gamma h(0) > 0 \}
\end{align}
and with affine and admissible state processes, and the set of initial curves has the decomposition $\fri = \partial \fri \oplus \frc$, where the boundary is given by
\begin{align*}
\partial \mathfrak{I} = \{ h \in \mathcal{D}(d/dx) : h(0) = 0 \text{ and } h'(0) > 0 \}.
\end{align*}
\end{corollary}

\begin{proof}
Noting that $\ell(\lambda \Lambda) = 0$, this is an immediate consequence of Proposition~\ref{prop-CIR}.
\end{proof}

\begin{remark}
Let $h \in \mathfrak{I}$ be arbitrary. According to Remark~\ref{rem-CIR} we can choose the short rate $r(0)$ as state process of the FDR, that is, the strong solution $r$ to the HJMM equation (\ref{HJMM}) with $r_0 = h$ is given by
\begin{align*}
r_t = \psi(t) + r_t(0) \cdot \lambda, \quad t \in \bbt
\end{align*}
for some time interval $\bbt \in \bbj$. In particular, we have $\mathbb{P}(r_t(0) \geq 0) = 1$ for all $t \in \mathbb{T}$. The expectation hypothesis (see, e.g., \cite[Lemma~7.2]{Filipovic}) implies that the initial curve $h$ satisfies
\begin{align*}
h(t) = \mathbb{E}_{\mathbb{P}^t}[r_t(0)] \geq 0 \quad \text{for all $t \in \mathbb{T}$,} 
\end{align*}
where $\mathbb{P}^t$ denotes the $t$-forward measure.
This is in accordance with the representation (\ref{initial-ex-1b}) of the set $\mathfrak{I}$ of initial curves, which shows that either $h(0) > 0$, or, otherwise, we have $h(0) = 0$ and $h'(0) > 0$.
\end{remark}

For our next example, we suppose that the function $\lambda$ in (\ref{lambda-ell}) is given by $\lambda(x) = e^{-\gamma x}$, $x \in \bbr_+$ for some constant $\gamma \in (0,\infty)$, and that the function $\ell$ in (\ref{lambda-ell}) satisfies $\ell(\lambda) = 1$ and $\ell(\lambda^2) = 0$. Then, according to Proposition~\ref{prop-qe}, the HJMM equation (\ref{HJMM}) cannot have an affine realization generated by some subspace with affine and admissible state processes. However, we will show that it admits an affine realization generated by a state space of the form $\frc \oplus U$ with $\dim \frc = 1$ and $\dim U = 1$, and with affine and admissible state processes.

\begin{proposition}
The HJMM equation (\ref{HJMM}) has an affine realization generated by $\mathfrak{C} \oplus U = \langle \lambda \rangle^+ \oplus \langle \lambda^2 \rangle$ with initial curves
\begin{align*}
\mathfrak{I} = \{ h \in \mathcal{D}(d/dx) : \ell(h) \geq 0 \text{ and } \ell(h' + \gamma \cdot h) > 0 \}
\end{align*}
and with affine and admissible state processes.
\end{proposition}

\begin{proof}
As noted in Remark~\ref{remark-orth-basis}, we may assume that $(\lambda,\lambda^2)$ is an orthonormal basis of $V$. Setting $G := {\rm ker}(\ell) \cap \langle \lambda^2 \rangle^{\perp}$, we have the direct sum decomposition $H = G \oplus V$, the corresponding projections are given by
\begin{align}\label{projections-2}
\Pi_G h = h - \ell(h) \lambda - \langle h,\lambda^2 \rangle_H \lambda^2 \quad \text{and} \quad \Pi_V h = \ell(h) \lambda + \langle h,\lambda^2 \rangle_H \lambda^2,
\end{align}
and we have
\begin{align}\label{G-oplus-C-U}
G \oplus \frc \oplus U = \{ h \in H : \ell(h) \geq 0 \}.
\end{align}
For each $g \in G$ we have (\ref{sigma-affine-parallel}) and the mapping (\ref{map-G-const}) is constant. Furthermore, we have
\begin{align*}
\frac{d}{dx} \lambda + S \sigma^2(\lambda) = \frac{d}{dx} \lambda + \alpha_{\rm HJM}(\lambda) = -\gamma \lambda + \rho^2 \lambda \Lambda = - \gamma \lambda + \rho^2 \frac{\lambda - \lambda^2}{\gamma} \in V,
\end{align*}
showing that condition (\ref{Ac-new}) is satisfied, and condition (\ref{cond-AR-3}) is fulfilled, because $\lambda^2(x) = e^{-2\gamma x}$, $x \in \bbr_+$. By (\ref{projections-2}), for each $h \in \cald(d/dx)$ we have
\begin{align*}
\Pi_V \frac{d}{dx} \Pi_G h &= \Pi_V \frac{d}{dx} ( h - \ell(h) \lambda - \langle h,\lambda^2 \rangle_H \lambda^2 )
= \Pi_V ( h' + \gamma \cdot \ell(h) \lambda + 2 \gamma \langle h,\lambda^2 \rangle_H \lambda^2 )
\\ &= \ell(h') \lambda + \langle h',\lambda^2 \rangle_H \lambda^2 + \gamma \cdot \ell(h) \lambda + 2 \gamma \langle h,\lambda^2 \rangle_H \lambda^2
\\ &= \ell(h' + \gamma \cdot h) \lambda + \langle h' + 2 \gamma \cdot h,\lambda^2 \rangle_H \lambda^2,
\end{align*}
and since $G \subset \ker(\ell)$, for each $g \in G$ we have
\begin{align*}
S \sigma^2(g) = \alpha_{\rm HJM}(g) = \rho^2 \ell(g) \lambda \Lambda = 0.
\end{align*}
Therefore, and taking into account (\ref{G-oplus-C-U}), applying Proposition~\ref{prop-suff} finishes the proof.
\end{proof}

\section{Linear SPDEs and examples from natural sciences}\label{sec-nat-sci}

In this section, we treat linear SPDEs
\begin{align}\label{SPDE-linear}
\left\{
\begin{array}{rcl}
dr_t & = & A r_t dt + \sigma(r_t) dW_t
\medskip
\\ r_0 & = & h_0
\end{array}
\right.
\end{align}
with continuous volatility $\sigma : H \to H^n$, and present some examples from natural sciences. The following two results essentially say that the linear SPDE (\ref{SPDE-linear}) admits an affine realization if and only if the volatility $\sigma$ is $A$-quasi-exponential; see, for example \cite[Thm.~5.6]{Tappe-affin-real} for a closely related result.

\begin{proposition}\label{prop-qe-pre-nec}
Suppose that the linear SPDE (\ref{SPDE-linear}) has an affine realization generated by some subspace with initial points $\cald(A)$. Then the volatility $\sigma$ is $A$-quasi-exponential.
\end{proposition}

\begin{proof}
There exists a finite dimensional subspace $U$ such that the linear SPDE (\ref{SPDE-linear}) has an affine realization generated by $U$ with initial points $\cald(A)$. By Remark~\ref{remark-add} we have $U \subset \cald(A)$ and $A(U) \subset U$. This yields $A_{\sigma} \subset U$, showing that $\sigma$ is $A$-quasi-exponential.
\end{proof}

\begin{proposition}\label{prop-qe-pre-suff}
Suppose that the volatility $\sigma$ is $A$-quasi-exponential and Lipschitz continuous. Then the linear SPDE (\ref{SPDE-linear}) has an affine realization generated by $A_{\sigma}$ with initial points $\cald(A)$.
\end{proposition}

\begin{proof}
Setting $U := A_{\sigma}$ we have $U \subset \cald(A)$ and $A(U) \subset U$. Thus, by Remark~\ref{remark-add} the linear SPDE (\ref{SPDE-linear}) has an affine realization generated by $A_{\sigma}$ with initial points $\cald(A)$.
\end{proof}

Now, we characterize when the linear SPDE (\ref{SPDE-linear}) has an affine realization with affine and admissible state processes.

\begin{proposition}\label{prop-qe-lin}
The following statements are equivalent:
\begin{enumerate}
\item[(i)] The linear SPDE (\ref{SPDE-linear}) has an affine realization generated by some subspace with initial points $\cald(A)$ and with affine and admissible state processes.

\item[(ii)] The linear SPDE (\ref{SPDE-linear}) has an affine realization generated by some subspace with initial points $\cald(A)$ and with affine state processes.

\item[(iii)] The volatility $\sigma$ is $A$-quasi-exponential, and for each $h \in H$ the mapping
\begin{align*}
v \mapsto \sigma^2(h+v) : A_{\sigma} \to L(A_{\sigma},L(A_{\sigma}))
\end{align*}
is constant.
\end{enumerate}
If the previous conditions are fulfilled, then the linear SPDE (\ref{SPDE-linear}) has an affine realization generated by $A_{\sigma}$ with initial points $\cald(A)$ and with affine and admissible state processes.
\end{proposition}

\begin{proof}
This follows from Propositions~\ref{prop-qe} and \ref{prop-qe-pre-nec}.
\end{proof}

\begin{corollary}\label{cor-U-2-lin}
Suppose that conditions (\ref{U-2-cond-1}) and (\ref{U-2-cond-2}) are fulfilled. Then the following statements are equivalent:
\begin{enumerate}
\item[(i)] The linear SPDE (\ref{SPDE-linear}) has an affine realization generated by some subspace with initial points $\cald(A)$ and with affine and admissible state processes.

\item[(ii)] The linear SPDE (\ref{SPDE-linear}) has an affine realization generated by some subspace with initial points $\cald(A)$ and with affine state processes.

\item[(iii)] The volatility $\sigma$ is $A$-quasi-exponential, and for each $h \in H$ the mapping 
\begin{align*}
v \mapsto \sigma(h + v) : A_{\sigma} \to A_{\sigma}^n 
\end{align*}
is constant.
\end{enumerate}
If the previous conditions are fulfilled, then the linear SPDE (\ref{SPDE-linear}) has an affine realization generated by $A_{\sigma}$ with initial points $\cald(A)$ and with affine and admissible state processes.
\end{corollary}

\begin{proof}
This is an immediate consequence of Propositions~\ref{prop-qe-lin} and \ref{prop-sigma-constant}.
\end{proof}

Here are some examples of SPDEs arising from natural sciences. For what follows, $\Delta$ denotes the Laplace operator.

\begin{example}
We consider the stochastic quantization of the free Euclidean quantum field (cf. \cite[Ex.~1.0.1]{Prevot-Roeckner})
\begin{align}\label{SPDE-quantum}
\left\{
\begin{array}{rcl}
dX_t & = & (\Delta - m^2) X_t dt + \sigma dW_t
\medskip
\\ X_0 & = & h_0,
\end{array}
\right.
\end{align}
where $m \in \bbr_+$ denotes ``mass'', and the volatility $\sigma \in H^n$ is constant. According to Proposition~\ref{prop-qe-lin}, the following statements are equivalent:
\begin{enumerate}
\item[(i)] The linear SPDE (\ref{SPDE-quantum}) has an affine realization generated by some subspace with initial points $\cald(\Delta)$ and with affine and admissible state processes.

\item[(ii)] The volatility $\sigma$ is $\Delta$-quasi-exponential.
\end{enumerate}
\end{example}

\begin{example}
We consider the stochastic cable equation (cf. \cite[Ex.~0.8]{Da_Prato})
\begin{align}\label{SPDE-cable}
\left\{
\begin{array}{rcl}
dV_t & = & \frac{1}{\tau} (\lambda^2 \Delta V_t - V_t) dt + \sigma dW_t
\medskip
\\ V_0 & = & h_0,
\end{array}
\right.
\end{align}
where $\lambda > 0$ denotes the length constant, $\tau > 0$ denotes the time constant of the electric cable, and the volatility $\sigma \in H^n$ is constant. According to Proposition~\ref{prop-qe-lin}, the following statements are equivalent:
\begin{enumerate}
\item[(i)] The linear SPDE (\ref{SPDE-cable}) has an affine realization generated by some subspace with initial points $\cald(\Delta)$ and with affine and admissible state processes.

\item[(ii)] The volatility $\sigma$ is $\Delta$-quasi-exponential.
\end{enumerate}
\end{example}

\begin{appendix}

\section{Convex cones and affine mappings}\label{app-affine}
 
The goal of this appendix is to provide the crucial results about convex cones and affine mappings, which we require for this paper. Throughout this section, let $H$ be a Hilbert space. Let $\frc$ be a finite dimensional proper convex cone, that is
\begin{align*}
\mathfrak{C} = \langle \lambda_1,\ldots,\lambda_m \rangle^+ := \bigg\{ \sum_{i=1}^m \alpha_i \lambda_i : \alpha_1,\ldots,\alpha_m \geq 0 \bigg\}
\end{align*}
with linearly independent $\lambda_1,\ldots,\lambda_m \in H$ for some $m \in \bbn_0$. We call $(\lambda_1,\ldots,\lambda_m)$ a basis of $\mathfrak{C}$. The basis $(\lambda_1,\ldots,\lambda_m)$ is called a normed basis, if $\| \lambda_i \|_H = 1$ for all $i=1,\ldots,m$. 

\begin{lemma}\label{lemma-unique-vectors}
Let $\lambda = (\lambda_1,\ldots,\lambda_m)$ and $\mu = (\mu_1,\ldots,\mu_m)$ be two bases of $\mathfrak{C}$. Then the following statements are true:
\begin{enumerate}
\item We have
\begin{align}\label{edges-equal}
\langle \lambda_1 \rangle^+ \cup \ldots \cup \langle \lambda_m \rangle^+ = \langle \mu_1 \rangle^+ \cup \ldots \cup \langle \mu_m \rangle^+.
\end{align}

\item Suppose the two bases $\lambda$ and $\mu$ are normed. Let $\alpha_i,\beta_i,\gamma_i,\delta_i \in \bbr$, $i=1,\ldots,m$ be such that
\begin{align}\label{inner-prep}
\sum_{i=1}^m \alpha_i \lambda_i = \sum_{i=1}^m \gamma_i \mu_i \quad \text{and} \quad \sum_{i=1}^m \beta_i \lambda_i = \sum_{i=1}^m \delta_i \mu_i.
\end{align}
Then we have
\begin{align}\label{inner-equal}
\sum_{i=1}^m \alpha_i \beta_i = \sum_{i=1}^m \gamma_i \delta_i.
\end{align}
\end{enumerate}
\end{lemma}

\begin{proof}
Let $M \in \bbr^{m \times m}$ be the matrix of the identity operator on the linear space $\langle \frc \rangle$ with respect to the bases $\lambda$ and $\mu$, that is, we have
\begin{align*}
\lambda_j = \sum_{i=1}^m M_{ij} \mu_i \quad \text{for all $j=1,\ldots,m$.}
\end{align*}
Then $M$ is nonnegative, that is $M_{ij} \geq 0$ for all $i,j = 1,\ldots,m$. Hence, according to \cite[Lemma~4.3, page~68]{Berman} there are $c_1,\ldots,c_m \in (0,\infty)$ and a permutation $\pi : \{ 1,\ldots,m \} \rightarrow \{ 1,\ldots,m \}$ such that
\begin{align*}
M = {\rm diag}(c_1,\ldots,c_m) \cdot \left( 
\begin{array}{ccc}
e_{\pi(1)} & \ldots & e_{\pi(m)}            
\end{array}
\right),
\end{align*}
where $e_1,\ldots,e_m \in \mathbb{R}^m$ denote the unit vectors in $\mathbb{R}^m$. Hence, we have
\begin{align*}
\lambda_j = c_{\pi(j)} \mu_{\pi(j)} \quad \text{for all $j=1,\ldots,m$,}
\end{align*}
which proves (\ref{edges-equal}). If the two bases $\lambda$ and $\mu$ are normed, then we even have
\begin{align*}
\lambda_j = \mu_{\pi(j)} \quad \text{for all $j=1,\ldots,m$.}
\end{align*}
Thus, if (\ref{inner-prep}) is fulfilled, then we have (\ref{inner-equal}).
\end{proof}

\begin{definition}\label{def-edges-C}
We introduce the following notions:
\begin{enumerate}
\item We define the \emph{edges of $\mathfrak{C}$} as
\begin{align*}
\partial \mathfrak{C} := \langle \lambda_1 \rangle^+ \cup \ldots \cup \langle \lambda_m \rangle^+,
\end{align*}
where $(\lambda_1,\ldots,\lambda_m)$ denotes a basis of $\mathfrak{C}$.

\item Let $c \in \partial \frc$ be arbitrary. If $c = 0$, then we define
\begin{align*}
\frc \ominus \langle c \rangle^+ := \frc,
\end{align*}
and otherwise, we define the new cone as
\begin{align*}
\frc \ominus \langle c \rangle^+ := \langle \lambda_i : i \in \{ 1,\ldots,m \} \setminus \{ j \} \rangle^+,
\end{align*}
where $(\lambda_1,\ldots,\lambda_m)$ denotes a basis of $\mathfrak{C}$ and $j \in \{ 1,\ldots,m \}$ is the unique index such that $c \in \langle \lambda_j \rangle^+$.
\end{enumerate}
\end{definition}

\begin{remark}
By virtue of Lemma~\ref{lemma-unique-vectors}, the definitions of the edges $\partial \mathfrak{C}$ and of the new cone $\frc \ominus \langle c \rangle^+$ do not depend on the choice of the basis.
\end{remark}

\begin{definition}
We define the inner product $\langle \cdot,\cdot \rangle_C$ as
\begin{align*}
\langle h,g \rangle_C := \sum_{i=1}^m \alpha_i \beta_i,
\end{align*}
where
\begin{align*}
h = \sum_{i=1}^m \alpha_i \lambda_i \quad \text{and} \quad g = \sum_{i=1}^m \beta_i \lambda_i,
\end{align*}
and $(\lambda_1,\ldots,\lambda_m)$ denotes a normed basis of $\mathfrak{C}$.
\end{definition}

\begin{remark}
By virtue of Lemma~\ref{lemma-unique-vectors}, the definition of the inner product $\langle \cdot,\cdot \rangle_C$ does not depend on the choice of the normed basis.
\end{remark}

Now, let $U \subset H$ be a finite dimensional subspace such that $C \cap U = \{ 0 \}$, where $C = \langle \frc \rangle$. We assume that the subspace $V = C \oplus U$ satisfies $\dim V \geq 1$.

\begin{definition}\label{def-inner-prod}
We define the inner product $\langle \cdot,\cdot \rangle_V$ as
\begin{align*}
\langle c_1 + u_1,c_2 + u_2 \rangle_V := \langle c_1,c_2 \rangle_C + \langle u_1,u_2 \rangle_H.
\end{align*}
\end{definition}

\begin{remark}
Note that $C = U^{\perp}$ and $U = C^{\perp}$, considered on the Hilbert space $(V,\langle \cdot,\cdot \rangle_V)$. 
\end{remark}

\begin{remark}\label{remark-orth-basis}
Let $\lambda = (\lambda,\ldots,\lambda_d)$ be a basis of $V$ such that $\frc = \langle \lambda,\ldots,\lambda_m \rangle^+$, where $m = \dim C$.
\begin{itemize}
\item There exists an inner product $( \cdot,\cdot )_H$ on $H$ such that $\langle \cdot,\cdot \rangle_H$ and $( \cdot,\cdot )_H$ generate equivalent norms on the Hilbert space $H$, the basis $\lambda$ is an orthonormal basis of $V$ with respect to $( \cdot,\cdot )_H$, and the inner product $\langle \cdot,\cdot \rangle_V$ constructed according to Definition~\ref{def-inner-prod} coincides with the restriction of $( \cdot,\cdot )_H$ to $V$. 

\item Consequently, we may assume, without loss of generality, that $\lambda$ is an orthonormal basis with respect to the original inner product $\langle \cdot,\cdot \rangle_H$, and that $\langle \cdot,\cdot \rangle_V$ coincides with the restriction of $\langle \cdot,\cdot \rangle_H$ to $V$.
\end{itemize}
\end{remark}

\begin{definition}\label{def-beta-adm}
A mapping $\beta : \mathfrak{C} \oplus U \rightarrow V$ is called \emph{inward pointing at boundary points of $\frc \oplus U$} (in short \emph{inward pointing}) if
\begin{align}\label{beta-inv}
\langle \beta(v),\eta \rangle_V \geq 0 \quad \text{for all $v \in \mathfrak{C} \oplus U$ and all $\eta \in \mathfrak{C}$ with $\langle v,\eta \rangle_V = 0$.}
\end{align}
\end{definition}

Now, let $\beta : \mathfrak{C} \oplus U \rightarrow V$ be an affine mapping. Then there are unique $\beta_1 \in V$ and $\beta_2 \in L(V)$ such that we have the decomposition
\begin{align}\label{nu-decomp}
\beta(v) = \beta_1 + \beta_2(v), \quad v \in \mathfrak{C} \oplus U.
\end{align}

\begin{proposition}\label{prop-char-affine-1}
The following statements are equivalent:
\begin{enumerate}
\item[(i)] $\beta$ is inward pointing.

\item[(ii)] We have
\begin{align}\label{nu-1}
\beta_1 &\in \mathfrak{C} \oplus U,
\\ \label{nu-2-C} \beta_2(c) &\in (\mathfrak{C} + \langle c \rangle) \oplus U, \quad c \in \partial \mathfrak{C},
\\ \label{nu-2-U} \beta_2(U) &\subset U.
\end{align}
\end{enumerate}
\end{proposition}

\begin{proof}
(i) $\Rightarrow$ (ii): Since $\beta$ is inward pointing, for all $c \in \mathfrak{C}$, $u \in U$ and all $\eta \in \mathfrak{C}$ with $\langle c,\eta \rangle_V = 0$ we have
\begin{align}\label{test-affine-drift}
\langle \beta_1,\eta \rangle_V + \langle \beta_2(c),\eta \rangle_V + \langle \beta_2(u),\eta \rangle_V \geq 0.
\end{align}
Taking $c = u = 0$ in (\ref{test-affine-drift}), we have
\begin{align*}
\langle \beta_1,\eta \rangle_V \geq 0 \quad \text{for all $\eta \in \mathfrak{C}$,} 
\end{align*}
showing (\ref{nu-1}). Moreover, taking $c = 0$ in (\ref{test-affine-drift}) we have
\begin{align*}
\langle \beta_1,\eta \rangle_V + \langle \beta_2(u),\eta \rangle_V \geq 0 \quad \text{for all $u \in U$ and all $\eta \in \mathfrak{C}$.}
\end{align*}
This implies
\begin{align*}
\langle \beta_2(u),\eta \rangle_V = 0 \quad \text{for all $u \in U$ and all $\eta \in C$,}
\end{align*}
showing (\ref{nu-2-U}). Now, let $c \in \partial \mathfrak{C}$ be arbitrary. Taking $u=0$ in (\ref{test-affine-drift}) we obtain
\begin{align*}
\langle \beta_1 + \beta_2(c),\eta \rangle_V \geq 0
\end{align*}
for all $\eta \in \mathfrak{C}$ with $\langle c,\eta \rangle_V = 0$, and hence
\begin{align*}
\beta_1 + \beta_2(c) \in (\mathfrak{C} + \langle c \rangle) \oplus U.
\end{align*}
Since $\beta_1 \in \mathfrak{C} \oplus U$, this implies (\ref{nu-2-C}).

\noindent(ii) $\Rightarrow$ (i): Let $v \in \mathfrak{C} \oplus U$ and $\eta \in \mathfrak{C}$ with $\langle v,\eta \rangle_V = 0$ be arbitrary. There exist unique elements $c \in \mathfrak{C}$ and $u \in U$ such that $v = c + u$. Moreover, there exist linearly independent elements $c_1,\ldots,c_p \in \partial \mathfrak{C}$ for some $p \in \{ 1,\ldots,m \}$ such that $c = \sum_{i=1}^p c_i$ and $\langle c_i,\eta \rangle_V = 0$ for all $i=1,\ldots,p$.
Therefore, by the decomposition (\ref{nu-decomp}) and (\ref{nu-1})--(\ref{nu-2-U}) we obtain
\begin{align*}
\langle \beta(v),\eta \rangle_V = \langle \beta_1,\eta \rangle_V + \sum_{i=1}^p \langle \beta_2(c_i),\eta \rangle_V + \langle \beta_2(u),\eta \rangle_V \geq 0,
\end{align*}
showing that $\beta$ is inward pointing.
\end{proof}

In the sequel, we fix a positive integer $n \in \bbn$.

\begin{definition}\label{def-par-1}
A mapping $\sigma : \mathfrak{C} \oplus U \rightarrow V^n$ is called \emph{parallel to the boundary at boundary points of $\frc \oplus U$} (in short \emph{parallel}) if for each $k = 1,\ldots,n$ we have
\begin{align}\label{sigma-inv}
\langle \sigma_k(v),\eta \rangle_V = 0 \quad \text{for all $v \in \mathfrak{C} \oplus U$ and all $\eta \in \mathfrak{C}$ with $\langle v,\eta \rangle_V = 0$.}
\end{align}
\end{definition}

For what follows, we denote by $e = (e_1,\ldots,e_n)$ the standard basis of $\bbr^n$.

\begin{definition}
For $\sigma \in V^n$ we define $\hat{\sigma} \in L(\bbr^n,V)$ by $\hat{\sigma} e_k := \sigma_k$ for $k=1,\ldots,n$.
\end{definition}

Note that the mapping $\sigma \mapsto \hat{\sigma}$ is an isomorphism from $V^n$ to $L(\bbr^n,V)$. In the sequel, we denote by $S^+(V) \subset L(V)$ the convex cone of all symmetric, nonnegative linear operators from $V$ to $V$.

\begin{definition}
For $\sigma \in V^n$ we define $\sigma^2 \in S^+(V)$ as $\sigma^2 := \hat{\sigma} \hat{\sigma}^*$, where the adjoint operator is defined with respect to the standard inner product on $\bbr^n$ and the inner product $\langle \cdot,\cdot \rangle_V$ from Definition~\ref{def-inner-prod}.
\end{definition}

\begin{definition}\label{def-sigma-adm}
A mapping $\sigma : \mathfrak{C} \oplus U \rightarrow V^n$ is called \emph{square-affine} if $\sigma^2 : \mathfrak{C} \oplus U \rightarrow S^+(V)$ is affine.
\end{definition}

\begin{definition}\label{def-par-2}
A mapping $T : \mathfrak{C} \oplus U \rightarrow S^+(V)$ is called \emph{parallel to the boundary at boundary points of $\frc \oplus U$} (in short \emph{parallel}) if
\begin{align*}
\langle (Tv)\eta,\eta \rangle_V = 0 \quad \text{for all $v \in \mathfrak{C} \oplus U$ and all $\eta \in \mathfrak{C}$ with $\langle v,\eta \rangle_V = 0$.}
\end{align*}
\end{definition}

\begin{lemma}\label{lemma-parallel}
For all $\sigma \in V^n$ and all $\eta \in V$ the following statements are equivalent:
\begin{enumerate}
\item[(i)] We have $\langle \sigma_k,\eta \rangle_V = 0$ for all $k=1,\ldots,n$.

\item[(ii)] We have $\hat{\sigma}^* \eta = 0$.

\item[(iii)] We have $\langle \sigma^2 \eta, \eta \rangle_V = 0$.
\end{enumerate}
\end{lemma}

\begin{proof}
For all $k = 1,\ldots,n$ we have
\begin{align*}
\langle \sigma_k,\eta \rangle_V = \langle \hat{\sigma} e_k,\eta \rangle_V = \langle e_k, \hat{\sigma}^* \eta \rangle_{\mathbb{R}^n},
\end{align*}
which proves (i) $\Leftrightarrow$ (ii). Moreover, we have
\begin{align*}
\| \hat{\sigma}^* \eta \|_{\bbr^n}^2 = \langle \hat{\sigma}^* \eta,\hat{\sigma}^* \eta \rangle_{\bbr^n} = \langle \hat{\sigma} \hat{\sigma}^* \eta, \eta \rangle_V = \langle \sigma^2 \eta, \eta \rangle_V,
\end{align*}
proving (ii) $\Leftrightarrow$ (iii).
\end{proof}

\begin{corollary}
For a mapping $\sigma : \mathfrak{C} \oplus U \rightarrow V^n$ the following statements are equivalent:
\begin{enumerate}
\item[(i)] $\sigma$ is parallel in the sense of Definition~\ref{def-par-1}.

\item[(ii)] $\sigma^2$ is parallel in the sense of Definition~\ref{def-par-2}.
\end{enumerate}
\end{corollary}

\begin{proof}
This is an immediate consequence of Lemma~\ref{lemma-parallel}.
\end{proof}

\begin{lemma}\label{lemma-symmetric-op}
For every $T \in S^+(V)$ the following statements are equivalent:
\begin{enumerate}
\item[(i)] We have $\langle Tc,c \rangle_V = 0$ for all $c \in \partial \mathfrak{C}$.

\item[(ii)] We have $\langle Tc,c \rangle_V = 0$ for all $c \in C$.

\item[(iii)] We have $T(C) \subset U$.

\item[(iv)] We have $T(U) \subset U$ and $C \subset {\rm ker}(T)$.

\item[(v)] We have $C \subset {\rm ker}(T)$.

\item[(vi)] We have $\partial \mathfrak{C} \subset {\rm ker}(T)$.
\end{enumerate}
\end{lemma}

\begin{proof}
(i) $\Rightarrow$ (ii): There exist an orthonormal basis $\{ f_1,\ldots,f_d \}$ of $V$ and eigenvalues $x_1,\ldots,x_d \geq 0$ of $T$ such that
\begin{align*}
Tv = \sum_{k=1}^d x_k \langle v,f_k \rangle_V f_k \quad \text{for all $v \in V$.}
\end{align*}
For each $c \in C$ we obtain
\begin{align*}
\langle Tc,c \rangle_V = \sum_{k=1}^d x_k |\langle c,f_k \rangle_V|^2.
\end{align*}
By assumption, we deduce that
\begin{align*}
\langle c,f_k \rangle_V = 0 \quad \text{for all $c \in \partial \mathfrak{C}$ and all $k = 1,\ldots,d$ with $x_k > 0$.}
\end{align*}
This gives us
\begin{align*}
\langle c,f_k \rangle_V = 0 \quad \text{for all $c \in C$ and all $k = 1,\ldots,d$ with $x_k > 0$,}
\end{align*}
and hence, we arrive at $\langle Tc,c \rangle_V = 0$ for all $c \in C$.

\noindent(ii) $\Rightarrow$ (iii): Let $c \in C$ be arbitrary. Then, by polarization, for all $\gamma \in C$ we have
\begin{align*}
\langle Tc,\gamma \rangle_V = \frac{1}{4} \big( \langle T(c+\gamma),c+\gamma \rangle_V - \langle T(c-\gamma),c-\gamma \rangle_V \big) = 0,
\end{align*}
showing that $Tc \in U$.

\noindent(iii) $\Rightarrow$ (iv): For all $c \in C$ and $u \in U$ we have
\begin{align*}
0 &\leq \langle T(c+u),c+u \rangle_V = \langle Tc,c \rangle_V + \langle Tc,u \rangle_V + \langle Tu,c \rangle_V + \langle Tu,u \rangle_V
\\ &= 2 \langle Tu,c \rangle_V + \langle Tu,u \rangle_V.
\end{align*}
Thus, for every $u \in U$ we obtain
\begin{align*}
\langle Tu,c \rangle_V = 0 \quad \text{for all $c \in C$,}
\end{align*}
showing that $Tu \in U$. Moreover, for every $c \in C$ we obtain
\begin{align*}
\langle Tc,u \rangle_V = 0 \quad \text{for all $u \in U$,}
\end{align*}
showing that $Tc \in C$. Therefore, and by assumption, we have $T(C) \subset C$ and $T(C) \subset U$, showing that $C \subset {\rm ker}(T)$.

\noindent The implications (iv) $\Rightarrow$ (v) $\Rightarrow$ (vi) $\Rightarrow$ (i) are obvious.
\end{proof}

Now, let $T : \mathfrak{C} \oplus U \rightarrow S^+(V)$ be an affine mapping. Then there are unique $T_1 \in S^+(V)$ and $T_2 \in L(V,L(V))$ with $T_2(\mathfrak{C} \oplus U) \subset S^+(V)$ such that we have the decomposition
\begin{align}\label{square-decomp}
T v &= T_1 + T_2 v, \quad v \in \mathfrak{C} \oplus U.
\end{align}

\begin{remark}\label{remark-square-affine}
Note that $T_2 u = 0$ for all $u \in U$, because $T_2(\mathfrak{C} \oplus U) \subset S^+(V)$.
\end{remark}

\begin{proposition}\label{prop-char-affine-2}
The following statements are equivalent:
\begin{enumerate}
\item[(i)] $T$ is parallel.

\item[(ii)] We have
\begin{align}\label{C-ker-T1} 
T_1 c &= 0, \quad c \in C,
\\ \label{U-ker-T2} T_2 u &= 0, \quad u \in U,
\\ \label{T2-c} \partial (\mathfrak{C} \ominus \langle c \rangle^+) &\subset {\rm ker}(T_2 c), \quad c \in \partial \mathfrak{C}.
\end{align}
\end{enumerate}
\end{proposition}

\begin{proof}
(i) $\Rightarrow$ (ii): Condition (\ref{U-ker-T2}) follows from Remark~\ref{remark-square-affine}. Since $T$ is parallel, for all $c \in \mathfrak{C}$, $u \in U$ and all $\eta \in \mathfrak{C}$ with $\langle c,\eta \rangle_V = 0$ we have
\begin{align}\label{T1-eta}
\langle T_1 \eta,\eta \rangle_V + \langle (T_2 c) \eta,\eta \rangle_V + \langle (T_2 u) \eta,\eta \rangle_V = 0.
\end{align}
Setting $c=u=0$ in (\ref{T1-eta}), we obtain
\begin{align*}
\langle T_1 \eta,\eta \rangle_V = 0 \quad \text{for all $\eta \in \mathfrak{C}$,}
\end{align*}
and hence, by Lemma~\ref{lemma-symmetric-op} we have (\ref{C-ker-T1}). Furthermore, by (\ref{C-ker-T1}), (\ref{U-ker-T2}) and (\ref{T1-eta}) we obtain
\begin{align*}
\langle (T_2 c)\eta,\eta \rangle_V = 0
\end{align*}
for all $c,\eta \in \mathfrak{C}$ with $\langle c,\eta \rangle_V = 0$. For every $c \in \partial \mathfrak{C}$ this yields
\begin{align*}
\langle (T_2 c)\eta,\eta \rangle_V = 0, \quad \eta \in \partial (\mathfrak{C} \ominus \langle c \rangle^+),
\end{align*}
and hence, by Lemma~\ref{lemma-symmetric-op} we obtain (\ref{T2-c}).

\noindent (ii) $\Rightarrow$ (i): Let $v \in \mathfrak{C} \oplus U$ and $\eta \in \mathfrak{C}$ with $\langle v,\eta \rangle_V = 0$ be arbitrary. There exist unique elements $c \in \mathfrak{C}$ and $u \in U$ such that $v = c + u$.
Moreover, there exist linearly independent elements $c_1,\ldots,c_p \in \partial \mathfrak{C}$ for some $p \in \{ 1,\ldots,m \}$ and linearly independent elements $\eta_1,\ldots,\eta_q \in \partial \mathfrak{C}$ for some $q \in \{ 1,\ldots,m \}$ 
such that $c = \sum_{i=1}^p c_i$, $\eta = \sum_{k=1}^q \eta_k$, and $\eta_k \in \partial (\mathfrak{C} \ominus \langle c_i \rangle^+)$ for all $i=1,\ldots,p$ and $k=1,\ldots,q$. Thus, by the decomposition (\ref{square-decomp}) and (\ref{C-ker-T1})--(\ref{T2-c}) we obtain
\begin{align*}
\langle (Tv) \eta,\eta \rangle_V &= \langle T_1 \eta,\eta \rangle_V + \langle (T_2 c) \eta,\eta \rangle_V + \langle (T_2 u) \eta,\eta \rangle_V
\\ &= \sum_{i=1}^p \sum_{k=1}^q \langle (T_2 c_i) \eta_k,\eta \rangle_V = 0,
\end{align*}
proving that $T$ is parallel.
\end{proof}

\begin{remark}
Note that for the canonical state space $\frc \oplus U = \bbr_+^m \times \bbr^{d-m}$ the conditions from Propositions~\ref{prop-char-affine-1} and \ref{prop-char-affine-2} correspond to the admissibility conditions for the local characteristics of affine processes, as, for example, defined in \cite{Filipovic-Mayerhofer}.
\end{remark}

For the rest of this appendix, we prepare further auxiliary results which we will need in this paper.

\begin{definition}\label{def-matrix-of-operator}
Let $X$ and $Y$ be two finite dimensional linear spaces with bases $\lambda = (\lambda_1,\ldots,\lambda_d)$ and $\mu = (\mu_1,\ldots,\mu_n)$, and let $T \in L(X,Y)$ be a linear operator.
\begin{enumerate}
\item We denote by $T^{\lambda,\mu} \in \bbr^{n \times d}$ the matrix of $T$ with respect to the bases $\lambda$ and $\mu$; that is, we have
\begin{align*}
T \lambda_j = \sum_{i=1}^n T_{ij}^{\lambda,\mu} \mu_i \quad \text{for all $j=1,\ldots,d$.}
\end{align*}
\item We denote by $\Psi^{\lambda,\mu} : L(X,Y) \to \bbr^{n \times d}$ the canonical isomorphism $\Psi^{\lambda,\mu} T = T^{\lambda,\mu}$.
\end{enumerate}
\end{definition}

\begin{definition}\label{def-sigma-coordinates}
For $\sigma \in V^n$ and a basis $\lambda = (\lambda_1,\ldots,\lambda_d)$ of $V$ we denote by $\sigma^{\lambda} \in \bbr^{n \times d}$ the matrix such that $\sigma = \sigma^{\lambda} \cdot \lambda$; that is
\begin{align*}
\sigma_i = \sum_{j=1}^d \sigma_{ij}^{\lambda} \lambda_j \quad \text{for all $i=1,\ldots,n$.}
\end{align*}
\end{definition}

\begin{lemma}\label{lemma-matrix-1}
For each $\sigma \in V^n$ and every basis $\lambda = (\lambda_1,\ldots,\lambda_d)$ of $V$ we have $\hat{\sigma}^{e,\lambda} = (\sigma^{\lambda})^{\top}$.
\end{lemma}

\begin{proof}
For each $j=1,\ldots,n$ we have
\begin{align*}
\hat{\sigma} e_j = \sigma_j = \sum_{i=1}^d \sigma_{ji}^{\lambda} \lambda_i = \sum_{i=1}^d (\sigma^{\lambda})_{ij}^{\top} \lambda_i,
\end{align*}
finishing the proof.
\end{proof}

There exists an orthonormal basis $\lambda = (\lambda_1,\ldots,\lambda_d)$ of $V$ with respect to $\langle \cdot,\cdot \rangle_V$ such that $\frc = \langle \lambda_1,\ldots,\lambda_m \rangle^+$, where $m = \dim C$. From now on, we fix such an orthonormal basis $\lambda$.

\begin{lemma}\label{lemma-matrices}
For each $\sigma \in V^n$ we have $(\sigma^2)^{\lambda,\lambda} = (\sigma^{\lambda})^{\top} \cdot \sigma^{\lambda}$.
\end{lemma}

\begin{proof}
By Lemma~\ref{lemma-matrix-1} we have
\begin{align*}
(\sigma^2)^{\lambda,\lambda} = (\hat{\sigma} \hat{\sigma}^*)^{\lambda,\lambda} = \hat{\sigma}^{e,\lambda} \cdot (\hat{\sigma}^*)^{\lambda,e} = \hat{\sigma}^{e,\lambda} \cdot (\hat{\sigma}^{e,\lambda})^{\top} = (\sigma^{\lambda})^{\top} \cdot \sigma^{\lambda}, 
\end{align*}
completing the proof.
\end{proof}

\begin{lemma}\label{lemma-extend-space}
Let $\sigma \in V^n$ be arbitrary, and let $W$ be a finite dimensional subspace such that $V \subset W$. Furthermore, let $\mu = (\mu_1,\ldots,\mu_p)$ be a basis of $W$ such that $\lambda_i = \mu_i$ for all $i=1,\ldots,d$. Then we have
\begin{align*}
(\sigma^2)^{\mu,\mu} = \left(
\begin{array}{cc}
(\sigma^2)^{\lambda,\lambda} & 0
\\ 0 & 0
\end{array}
\right).
\end{align*}
\end{lemma}

\begin{proof}
Note that the matrix $\sigma^{\mu} \in \bbr^{n \times p}$ is given by
\begin{align*}
\sigma^{\mu} = \left(
\begin{array}{cc}
\sigma^{\lambda} & 0
\end{array}
\right).
\end{align*}
Therefore, by Lemma~\ref{lemma-matrices} we obtain
\begin{align*}
(\sigma^2)^{\mu,\mu} &= (\sigma^{\mu})^{\top} \cdot \sigma^{\mu} = \left(
\begin{array}{c}
(\sigma^{\lambda})^{\top}
\\ 0
\end{array}
\right) \cdot \left(
\begin{array}{cc}
\sigma^{\lambda} & 0
\end{array}
\right) 
\\ &= \left(
\begin{array}{cc}
(\sigma^{\lambda})^{\top} \cdot \sigma^{\lambda} & 0
\\ 0 & 0
\end{array}
\right) = \left(
\begin{array}{cc}
(\sigma^2)^{\lambda,\lambda} & 0
\\ 0 & 0
\end{array}
\right),
\end{align*}
completing the proof.
\end{proof}

\begin{lemma}\label{lemma-extend-sigma}
Let $\sigma \in V^n$ and $m \in \bbn$ with $n \leq m$ be arbitrary. We define $\tau \in V^m$ as $\tau_k := \sigma_k$ for $k=1,\ldots,n$ and $\tau_k := 0$ for $k=n+1,\ldots,m$. Then we have $\sigma^2 = \tau^2$.
\end{lemma}

\begin{proof}
Note that the matrix $\tau^{\lambda} \in \bbr^{m \times d}$ is given by
\begin{align*}
\tau^{\lambda} = \left(
\begin{array}{c}
\sigma^{\lambda}
\\ 0
\end{array}
\right).
\end{align*}
Therefore, by Lemma~\ref{lemma-matrices} we obtain
\begin{align*}
(\tau^2)^{\lambda,\lambda} = (\tau^{\lambda})^{\top} \cdot \tau^{\lambda} = \left( 
\begin{array}{cc}
(\sigma^{\lambda})^{\top} & 0 
\end{array}
\right) \cdot \left(
\begin{array}{c}
\sigma^{\lambda}
\\ 0
\end{array}
\right)
= (\sigma^{\lambda})^{\top} \cdot \sigma^{\lambda} = (\sigma^2)^{\lambda,\lambda},
\end{align*}
which proves $\tau^2 = \sigma^2$.
\end{proof}

\begin{proposition}\label{prop-sigma-constant}
Let $E \subset H$ be a subset, and let $\sigma : E \to V^n$ be a continuous mapping such that $\sigma^2 : E \to L(V)$ is constant, and we have
\begin{align}\label{sigma-const-1}
\dim \langle \sigma_k(E) \rangle &\leq 1 \quad \text{for all $k = 1,\ldots,n$} \quad \text{and}
\\ \label{sigma-const-2} \langle \sigma_k(E) \rangle \cap \langle \sigma_l(E) \rangle &= \{ 0 \} \quad \text{for all $k,l = 1,\ldots,n$ with $k \neq l$.}
\end{align}
Then $\sigma$ is constant, too. 
\end{proposition}

\begin{proof}
By (\ref{sigma-const-1}) and (\ref{sigma-const-2}) there exists a basis $\mu = (\mu_1,\ldots,\mu_d)$ of $V$ such that $n \leq d$ and $\sigma_k(E) \in \langle \mu_k \rangle$ for all $k=1,\ldots,n$. We define $\tau : E \to V^d$ as $\tau_k := \sigma_k$ for $k=1,\ldots,n$ and $\tau_k := 0$ for $k=n+1,\ldots,d$. Then, there exist continuous functions $\Phi_1,\ldots,\Phi_d : V \to \bbr$ such that
\begin{align*}
\tau^{\mu} = {\rm diag}(\Phi_1,\ldots,\Phi_d).
\end{align*}
Denoting by $f = (f_1,\ldots,f_d)$ the canonical orthonormal basis of $\bbr^d$, by Lemma~\ref{lemma-matrix-1} we have
\begin{align*}
\hat{\tau}^{f,\lambda} = ({\rm Id}^{\lambda,\mu})^{-1} \cdot \hat{\tau}^{f,\mu} = ({\rm Id}^{\lambda,\mu})^{-1} \cdot (\tau^{\mu})^{\top} = ({\rm Id}^{\lambda,\mu})^{-1} \cdot {\rm diag}(\Phi_1,\ldots,\Phi_d),
\end{align*}
and hence, by Lemma~\ref{lemma-extend-sigma} we obtain
\begin{align*}
(\sigma^2)^{\lambda,\lambda} = (\tau^2)^{\lambda,\lambda} = \hat{\tau}^{f,\lambda} \cdot (\hat{\tau}^{f,\lambda})^{\top} = ({\rm Id}^{\lambda,\mu})^{-1} \cdot {\rm diag}(\Phi_1^2,\ldots,\Phi_d^2) \cdot ({\rm Id}^{\lambda,\mu})^{-\top}.
\end{align*}
Since $\sigma^2$ is constant, we deduce that $\Phi_1^2,\ldots,\Phi_d^2$ are constant. Since $\Phi_1,\ldots,\Phi_d$ are continuous, we deduce that $\tau^{\mu}$ is constant. Consequently, the mapping $\sigma$ is constant, too.
\end{proof}

\end{appendix}

\end{document}